\documentclass[a4paper,leqno]{amsart}

\usepackage[T1]{fontenc}
\usepackage[utf8]{inputenc}

\usepackage[ngerman,english]{babel}

\usepackage[colorlinks=true, pdfstartview=FitV, linkcolor=blue, citecolor=blue, urlcolor=blue]{hyperref}

\usepackage{amssymb,amsthm,mathtools,mathrsfs,extarrows,bbm}

\usepackage{tikz-cd}
\usetikzlibrary{babel}

\usepackage[paper=a4paper,left=25mm,right=25mm, bottom=40mm]{geometry}

\theoremstyle{definition}
\newtheorem{defi}{Definition}[section]
\newtheorem{notation}[defi]{Notation}
\newtheorem*{ack}{Acknowledgements}
\theoremstyle{plain}
\newtheorem{lemma}[defi]{Lemma}
\newtheorem{prop}[defi]{Proposition}

\newtheorem{thm}[defi]{Theorem}
\newtheorem{example}[defi]{Example}

\numberwithin{equation}{section}

\newcommand{\N}{\mathbb{N}}
\newcommand{\Z}{\mathbb{Z}}

\newcommand{\R}{\mathbb{R}}
\newcommand{\1}{\mathbbm{1}}
\newcommand{\Oe}{{\Omega_\e}}
\newcommand{\Od}{{\Omega^{(\delta)}}}
\newcommand{\hOe}{{\hat{\Omega}_\e}}
\newcommand{\T}{\mathcal{T}}
\newcommand{\e}{\varepsilon}
\newcommand{\es}{{\varepsilon'}}
\newcommand{\xeps}{\frac{x}{\e}}
\newcommand{\xxeps}{\left(x,\xeps\right)}
\newcommand{\xy}{(x,y)}
\newcommand{\intO}{\int\limits_\Omega}
\newcommand{\intOe}{\int\limits_\Oe}
\newcommand{\inthOe}{\int\limits_{\hat{\Omega}_\e}}
\newcommand{\intY}{\int\limits_Y}
\newcommand{\intYp}{\int\limits_{\Yp}}
\newcommand{\intYpx}{\int\limits_{\Ypx}}
\newcommand{\intOY}{\intO \!\! \intY}
\newcommand{\intOYp}{\intO \!\intYp}
\newcommand{\intOYpx}{\intO \! \intYpx}
\newcommand{\norm}[2]{\left|\left|#1\right| \right|_{#2}}
\newcommand{\Yp}{{Y^*}}
\newcommand{\Ypx}{{Y^*_x}}
\newcommand{\hu}{{\hat{u}}}
\newcommand{\hue}{{\hat{u}_\e}}
\newcommand{\hp}{{\hat{\varphi}}}
\newcommand{\pe}{\varphi_{\e,\psi_\e}}
\newcommand{\pO}{\varphi_{\psi_0}}
\newcommand{\TscD}{\overset{D}{\rightharpoonup}\mathrel{\mspace{-15mu}}\rightharpoonup}
\newcommand{\TscW}[1]{\overset{#1}{\rightharpoonup}\mathrel{\mspace{-15mu}}\rightharpoonup}
\newcommand{\TscS}[1]{\overset{#1}{\rightarrow}\mathrel{\mspace{-15mu}}\rightarrow}
\newcommand{\Tscs}[1]{\overset{< #1}{\rightarrow}\mathrel{\mspace{-15mu}}\rightarrow}
\newcommand{\Tscw}[1]{\overset{< #1}{\rightharpoonup}\mathrel{\mspace{-15mu}}\rightharpoonup}

\title[The two-scale transformation method]{The two-scale transformation method}

\author[D.~Wiedemann]{David Wiedemann}
\address[D.~Wiedemann]{Fakultät f\"ur Mathematik, Universit\"at Augsburg, 86135 Augsburg, Germany}
\email{david.wiedemann@math.uni-augsburg.de}
\thanks{The author was partially supported by a doctoral scholarship of \emph{Studienstiftung des deutschen Volkes}.}

\date{June 25, 2021}

\begin{document}	
\begin{abstract}
We prove the two-scale transformation method which allows rigorous homogenisation of problems defined on locally periodic domains by transformation on periodic domains. 
The idea to consider periodic substitute problems was originally proposed by M.~A.~Peter for the homogenisation on evolving microstructure and is applied in several works. However, only the homogenisation of the periodic substitute problems was proven, whereas the method itself was just postulated (i.e.~the equivalence to the homogenisation of the actual problem had to be assumed). In this work, we develop this idea further and formulate a rigorous two-scale convergence concept for microscopic transformation to prove this method. 
Moreover, we show a new two-scale transformation rule for gradients which allows to derive new limit problems that are now transformationally independent.
\end{abstract}	
\maketitle
\tableofcontents

\section{Introduction}
Periodic homogenisation allows to derive effective macroscopic models for processes described on a fine heterogeneous microscopic structure as for example given in composite materials or porous media. Thus, effective physical, biological or geological models, like the Darcy law, can be derived rigorously.
The main assumption for periodic homogenisation is the microscopic periodic structure, which can be scaled by a parameter $\e>0$ and arbitrarily refined in a limit process.
However, this is too restrictive in many applications because it cannot capture local microscopic varieties or temporal changes of the microstructure, which can have a considerable impact.
For processes in which the periodic structure is only given by coefficients, as for example in composite materials, this assumption can be weakened. There, the convergence theory can handle coefficients $A_\e(t,x) = A\left(t,x,\frac{x}{\e}\right)$, which can capture spatial or temporal changes in the microstructure.
However, if the microstructure is also given by the domain $\Oe \subset \Omega \subset \R^N$ itself, as for example in porous media problems, the procedure can not be transferred directly. Instead, special compactness results are required. These, and often also the derivation of the solutions' uniform estimates, depend largely on the strict periodic structure (cf. \cite{All92}, \cite{All96}, \cite{Tar79}).

In order to overcome this strongly restricting microscopic periodicity, M.~A.~Peter proposed the following method in \cite{Pet07a}.
Instead of homogenising the actual problem, he transformed it into a substitute problem on a periodic domain. There, the domain's local periodicity becomes a local periodicity of the coefficients, which can be handled by the two-scale convergence. Then, the homogenised equation can be transformed back to an associated evolving domain.
However, the homogenisation of the substitute problem is a priori not equivalent to the homogenisation of the actual problem, which would mean that \eqref{eq:diagram} commutes. Therefore, the method itself was only proposed and has not been proven until now. Nevertheless, this method found wide application -- in the sense that the back-transformations is done formally and only the homogenisation of the substitute problems is proven --  since it allows to consider many interesting problems, particularly on domains evolving in time (see \cite{Pet07b}, \cite{Pet09a}, \cite{Pet09b}, \cite{EM17}, \cite{GNP21}).

In this work, we develop the idea of \cite{Pet07a} further, introduce a two-scale concept for \textit{locally periodic domains} and formulate a rigorous two-scale convergence concept for this transformation method. Thus, we can show that these transformations actually commute with the two-scale convergence. Hence, we call it \textit{two-scale transformation method}. This proves the method and also allows to show that \eqref{eq:diagram} commutes.
Furthermore, we prove a new two-scale transformation rule for gradients. Thus, we can improve the result for the limit problem in the case of a slow process (flux scaling by $\e^0$) significantly. Previously, the formal back-transformation could only tackle the homogenised problem and not the two-scale limit problem. As a consequence, the result still depended on the chosen transformation. With the help of the new gradient transformation rule, we can directly transform the two-scale limit equations back and derive a new (equivalent) homogenised problem which is independent of the transformation.

Moreover, the results developed in this work allow to translate two-scale compactness results from periodic domains to locally periodic domains. Consequently a direct homogenisation of the actual problem becomes possible if it is defined on locally periodic domains.
\begin{equation}\label{eq:diagram}
\hspace{-1.5cm}
\begin{tikzcd}[row sep=1cm, column sep = 6.5cm]
\textrm{microproblem} \arrow[r, "\textrm{homogenisation on locally periodic domains}"] \arrow[d, "\textrm{transformation}" ] & \textrm{macroproblem} 
\arrow[d, leftarrow, "\textrm{back-transformation}"] \\
\textrm{transformed microproblem}
\arrow[r, "\textrm{homogenisation on periodic domains}"]& \textrm{transformed macroproblem}
\end{tikzcd}
\end{equation}

This paper is arranged as follows: In Section \ref{sec:TwoScale-Unfolding}, we recap the two-scale convergence (cf.~\cite{All92}) as well as the unfolding operator (cf.~\cite{CDG08}) and state results about them, which we employ to prove the two-scale transformation method.
In Section \ref{sec:TrafoMethod}, we introduce the two-scale transformation method and formulate the assumptions on the locally periodic domain. Then, we show that the two-scale convergence and the two-scale transformation commute, which allows to perform \eqref{eq:diagram} rigorously. Moreover, we derive the new two-scale transformation rule for gradients, which improves the results of the back-transformation.
In Section \ref{sec:HomogenisationProblem}, we demonstrate the method by homogenising the following diffusion process on locally periodic domains $\Oe\subset \Omega$, which are defined in Section \ref{sec:TrafoMethod}. There, we consider the case of a fast flux ($l=0$) as well as the case of a slow flux ($l=2$).
Let $A_\e\in L^\infty(\Oe)^{N \times N}$ be  bounded and uniformly coercive, i.e.~there exist $C,\alpha >0$ such that $\norm{A_\e}{L^\infty(\Oe)} \leq C$ and $ \xi^\top A_\e(x) \xi \geq \alpha \norm{\xi}{}^2$ for every $\e >0$, a.e.~$x \in \Oe$ and every~$\xi \in \R^N$ and let $f_\e \in L^2(\Oe)$. Then, find $u_\e \in H^1(\Oe)$ such that
\begin{align}\label{eq:weakE}
\intOe \e^l A_\e(x)\nabla u_\e(x) \cdot \nabla \varphi(x) + u_\e(x) \varphi(x) dx = \intOe f_\e(x) \varphi(x) dx
\end{align}
for all $\varphi \in H^1(\Oe)$.
First, we transform \eqref{eq:weakE} on the periodic reference domain and show how to derive uniform estimates for the substitute problem.
Afterwards, we pass to the limit $\e\to0$ in the periodic substitute problem. By using the two-scale transformation method and particularly the new transformation rule for the gradients, we transform the limit problems back. The results are the homogenised problem \eqref{eq:weak-u0}  (in the case of $l=0$) and the two-scale problem \eqref{eq:weakForm-two-scale-limit-2} (in the case of $l=2$).

In the following, we use $C>0$ as generic constant, which is independent of $\e$. Let $(\e_n)_{n \in \N}$ be a fixed sequence of positive real numbers converging to $0$ (when it is clear from the context, we omit the subscript~$n$). Moreover, we write $\es$ for a subsequence  $(u_{n_k})_{k\in \N}$ of $\e$.

\section{Two-scale convergence and the unfolding operator}\label{sec:TwoScale-Unfolding}
For the homogenisation of \eqref{eq:weakE} and the assumption on the domain $\Oe$, we use the two-scale convergence (cf.~\cite{All92}, \cite{LNW02}).
Let $\Omega \subset \R^N$ be a bounded open set and, for simplicity, let $Y \coloneqq [0,1]^N$ denote the reference cell. Nevertheless, all the arguments can be transferred to arbitrary reference parallelotopes $Y\subset \R^N$.

\begin{defi}\label{def:T-Sc-D}
We say that a sequence $u_\e$ in $L^1(\Omega)$ two-scale converges distributionally to $u_0 \!\in L^1(\Omega \times Y)$ if
\begin{align}\label{eq:T-Sc-D}
\lim\limits_{\e\to 0}\intO u_\e(x) \varphi_\# \xxeps dx = \intOY u_0\xy \varphi\xy dydx
\end{align}
for every $\varphi \in D(\Omega;C^\infty_\#(Y))$. We write $u_\e \TscD u_0$.
\end{defi}
\begin{defi}\label{def:TwoScale}
Let $p,q \in (1,\infty)$ with $\frac{1}{p} + \frac{1}{q}=1$. We say that a sequence $u_\e$ in $L^p(\Omega)$ two-scale converges weakly to  $u_0 \in L^p(\Omega \times Y)$ if
\begin{align}\label{eq:defTwoScale}
\lim\limits_{\e\to 0}\intO u_\e(x) \varphi \xxeps dx = \intOY u_0\xy \varphi\xy dydx
\end{align}
for every $\varphi \in L^q(\Omega; C_\#(Y))$. We write $u_\e \TscW{p} u_0$.
\end{defi}

The main compactness result for the two-scale convergence is Theorem \ref{thm:T-SC-LpComp} (see \cite[Theorem 7]{LNW02}).
\begin{thm}\label{thm:T-SC-LpComp}
Let $p \in (1,\infty)$ and let $u_\e$ be a bounded sequence in $L^p(\Omega)$. Then, there exists a subsequence $\es$ and $u_0 \in L^p(\Omega \times Y)$ such that  $u_\es \TscW{p} u_0$.
\end{thm}
This compactness result can be improved for sequences of weakly differentiable functions by the following two standard two-scale compactness results.
\begin{thm}\label{thm:T-SC-W1pComp}
Let $p \in (1,\infty)$ and let $u_\e$ be a bounded sequence in $W^{1,p}(\Omega)$. Then,
there exists a subsequence $\es$, $u_0 \in W^{1,p}(\Omega)$ and $u_1 \in L^p(\Omega;W^{1,p}(Y)/\R)$ such that $u_\es \TscW{p}u_0$ and $\nabla u_\es \TscW{p} \nabla_x u_0 + \nabla_y u_1$.
\end{thm}
\begin{thm}\label{thm:T-SC-W1pComp2} 
Let $p \in (1,\infty)$ and let $u_\e$ be a sequence in $W^{1,p}(\Omega)$ such that \mbox{$\norm{u_\e}{L^p(\Omega)} \! + \! \e\norm{\nabla u_\e}{L^p(\Omega)} \!\leq C$}. Then, there exists a subsequence $\es$ and $u_0 \in L^p(\Omega;W^{1,p}_\#(Y))$ such that $u_\es \TscW{p} u_0$ and $\nabla u_\es \TscW{p}\nabla_y u_0$.
\end{thm}
Testing \eqref{eq:defTwoScale} with functions $\varphi \in L^q(\Omega)$ shows a relation between the weak two-scale convergence and the weak convergence in $L^p(\Omega)$.
\begin{prop}\label{prop:T-SC-Bound}
Let $p \in (1,\infty)$. Let $u_\e$ be a sequence in $L^p(\Omega)$ and $u_0 \in L^p(\Omega \times Y)$ such that $u_\e \TscW{p} u_0$. Then, $u_\e$ converges weakly in $L^p(\Omega)$ to $u(\cdot_x) = \intY u_0(\cdot_x,y) dy$ and the sequence $u_\e$ is bounded in $L^p(\Omega)$.
\end{prop}
On the other hand, if the sequence is bounded, the set of test functions can be reduced to smooth and, in a certain way, dense test functions, similar to the weak $L^p$-convergence (cf.~\cite[Proposition 1]{LNW02})
\begin{prop}\label{prop:T-SC-SmoothTestF}
Let $p \in (1,\infty)$. Let $u_\e$ be a bounded sequence in $L^p(\Omega)$ and $u_0 \in L^p(\Omega \times Y)$ such that $u_\e \TscD u_0$, then $u_\e \TscW{p} u_0$.
\end{prop}

\color{black} 
One desirable property of the two-scale convergence is that the product and the limit process commute (i.e.~$u_\e v_\e \to u_0v_0$). The well known fact that in uniformly convex Banach spaces the weak convergence plus the convergence of the norms are equivalent to the strong convergence motivates the following definition of the strong two-scale convergence.
\begin{defi}
Let $p \in (1,\infty)$. We say that a sequence $u_\e$ in $L^p(\Omega)$ two-scale converges strongly to $u_0 \in L^p(\Omega \times Y)$ if $u_\e \TscW{p}u_0$ and $\lim\limits_{\e \to 0} \norm{u_\e}{L^p(\Omega)} = \norm{u_0}{L^p(\Omega)}$. We write $u_\e \TscS{p} u_0$.
\end{defi}
\begin{thm}\label{thm:T-SC-Prod}
Let $1 <p,q,r< \infty$ with $\frac{1}{p}+ \frac{1}{q} = \frac{1}{r}$. Let $u_\e$ be a sequence in $L^p(\Omega)$ which two-scale converges strongly to $u_0 \in L^p(\Omega \times Y)$ and let $v_\e$ be a sequence in $L^q(\Omega)$ which  two-scale converges weakly (resp.~strongly) to $v_0 \in L^q(\Omega \times Y)$. Then, $u_\e v_\e$ is a sequence of functions in $L^r(\Omega)$ which two-scale converges weakly (resp.~strongly) to $u_0 v_0 \in L^r(\Omega \times Y)$.
\end{thm}
Using the unfolding operator $\T_\e$, which was introduced in \cite{CDG02}, two-scale convergence can be translated into convergence in $L^p(\Omega \times Y)$. Thus, we can give a brief proof of Theorem \ref{thm:T-SC-Prod} later. 

In order to simplify the proofs of the two-scale convergence method, we introduce the following notations.
\begin{notation}
Let $p \in (1,\infty]$. If $u_\e$ is a sequence in $L^q(\Omega)$ and $u_0 \in L^q(\Omega \times Y)$ for every $q\in (1,p)$ such that $u_\e \TscW{q} u_0$ for every $q\in (1,p)$, we write $u_\e \Tscw{p} u_0$. If additionally $u_\e \TscS{q} u_0$, we write $u_\e \Tscs{p} u_0$.
\end{notation}
%
\begin{notation}\label{def:Unfolding}
Let $\Omega \subset \R^N$ and $x = \sum\limits_{i = 1}^N x_i e_i \in \R^N$, where $e_i$ denotes the euclidean unit vectors, then let
\vspace{-0.3cm}
\begin{align*}
&[x]_Y \coloneqq \sum\limits_{i = 1}^N \lfloor x_i \rfloor e_i \ , \hspace{0.7cm}
\{x\}_Y \coloneqq x-[x]_Y \ , \hspace{0.7cm}
[x]_{\e,Y} \coloneqq \e\left[\xeps\right]_Y \ , \hspace{0.7cm}
\{x\}_{\e,Y} \coloneqq \left\{\xeps\right\}_Y \ ,\hspace{0.3cm}
\\
&I_\e \coloneqq \{k \in \e \Z^N, k + \e Y \subset \overline{\Omega}\}\ , \hspace{0.7cm}
\tilde{\Omega}_\e\coloneqq \operatorname{int}\Big( \bigcup\limits_{k \in I_\e} k + \e Y \Big)\ , \hspace{0.7cm}
\Lambda_\e = \Omega \setminus \tilde{\Omega}_\e.
\end{align*}	
\end{notation}
\begin{defi}
Let $1 \leq p \leq \infty$. We define the unfolding operator $\T_\e : L^p(\Omega) \rightarrow L^p(\Omega \times Y)$ by
\begin{align}
\T_\e(\varphi) \xy \coloneqq \begin{cases}
\varphi( [x]_{\e,Y} + \e y) &\textrm{ for a.e. }\xy \in \tilde{\Omega}_\e \times Y,
\\
\varphi(x) &\textrm{ for a.e. }\xy \in \Lambda_\e \times Y.
\end{cases}
\end{align}
\end{defi}
Note that we have defined $\T_\e(\varphi)\xy = \varphi(x)$ on the cells that are not completely included in $\Omega$ and not $\T_\e(\varphi)\xy = 0$ as in \cite{CDG08}. By this slight modification, $\T_\e$ becomes isometric (cf. Theorem \ref{thm:NormEqualityUnfold}). Thus, we can not only translate between the two-scale convergence of $u_\e$ and the weak convergence of $\T_\e(u_\e)$ in $L^p(\Omega \times Y)$, as shown in \cite{CDG08}, but we can also translate between the strong two-scale convergence and the strong convergence in $L^p(\Omega \times Y)$.
\begin{thm}\label{thm:NormEqualityUnfold}
Let $\varphi \in L^p(\Omega)$ for $1 \leq p \leq \infty$. Then 
\begin{align}\label{eq:EqualityIntegralT}
&\intOY \T_\e (\varphi)\xy dydx = \intO \varphi(x) dx,
\\\label{eq:EqualityNormT}
&\norm{\T_\e (\varphi)}{L^p(\Omega \times Y)}=\norm{\varphi}{L^p(\Omega)}.
\end{align}
\end{thm}
\begin{proof}
We split the integral on $\tilde{\Omega}_\e \times Y$ and $\Lambda_\e \times Y$ so that
\begin{align*}
\intOY \T_\e (\varphi)\xy dydx 
=
\sum\limits_{k \in I_\e}\int\limits_{k + \e Y} \intY \varphi([x]_{\e,Y} + \e y) dydx + \int\limits_{\Lambda_\e \times Y} \varphi(x) dx.
\end{align*}
Since $[x]_{\e,Y} = k$ on each cell $k + \e Y$, we obtain
\begin{align*}
\int\limits_{k + \e Y} \intY \varphi([x]_{\e,Y} + \e y) dydx
=
\int\limits_{k + \e Y} \intY \varphi(k + \e y) dydx
=
|\e Y| \intY \varphi(k + \e y) dy
=
\int\limits_{k + \e Y} \varphi(x) dx.
\end{align*}
Combining these two equations yields
\begin{align*}
\intOY \T_\e (\varphi)\xy dydx 
=
\sum\limits_{k \in I_\e}\int\limits_{k + \e Y} \varphi(x) dx + \int\limits_{\Lambda_\e} \varphi(x) dx
=
\intO \varphi(x) dx.
\end{align*}
Since $|\T_\e(\varphi)|^p= \T_\e(|\varphi|^p)$, \eqref{eq:EqualityNormT} follows for $p <\infty$ by applying \eqref{eq:EqualityIntegralT} to $|\varphi|^p$. For $p = \infty$, \eqref{eq:EqualityNormT} follows directly from the definition of $\T_\e$.
\end{proof}

\begin{thm}\label{thm:T-SC-Equi-Unf}
Let $u_\e$ be a sequence in $L^p(\Omega)$ and $u_0 \in L^p(\Omega \times Y)$ with $1 <p<\infty$. Then, the following statements hold:
\begin{enumerate}
\item \label{item:DistEqui} $u_\e \TscD u_0$ if and only if $\T_\e(u_\e) \varphi \rightarrow u_0 \varphi$  in $L^1(\Omega \times Y)$ for every $\varphi \in D(\Omega;C^\infty_\#(Y))$,
\item \label{item:WeakEqui} $u_\e \TscW{p} u_0$ if and only if $\T_\e(u_\e) \rightharpoonup u_0$  in $L^p(\Omega \times Y)$,
\item \label{item:StrongEqui} $u_\e \TscS{p} u_0$ if and only if $\T_\e(u_\e) \rightarrow u_0$  in $L^p(\Omega \times Y)$. 
\end{enumerate}
\end{thm}
\begin{proof}
In order to prove the first equivalence, it is enough to show that
\begin{align}\label{eq:limue=limTue}
\lim\limits_{\e \to 0} \intO u_\e(x) \varphi\xxeps dx
= 
\lim\limits_{\e \to 0} \intOY \T_\e(u_\e)\xy \T_\e\left(\varphi\left(\cdot_x, \frac{\cdot_x}{\e}\right)\right) \!\xy dy dx
=
\lim\limits_{\e \to 0} \intOY \T_\e(u_\e)\xy \varphi\xy dx
\end{align}
for every smooth test function $\varphi \in D(\Omega;C^\infty_\#(Y))$.
The first equality in \eqref{eq:limue=limTue} follows directly from the definition of $\T_\e$ and \eqref{eq:EqualityIntegralT}. For the second equality, it is enough to show that $\T_\e(\varphi\left(\cdot_x, \frac{\cdot_x}{\e}\right))$ converges strongly to $\varphi$ in $L^q(\Omega \times Y)$ for $ \frac{1}{p}+ \frac{1}{q} =1$. We note that for every $x \in \Omega$ there exists $\e_0 >0$ small enough such that $x \in \tilde{\Omega}_\e$ for every $0<\e <\e_0$. Thus, we obtain the pointwise convergence
\begin{align*}
\T_\e\left(\varphi\left(\cdot_x, \frac{\cdot_x}{\e}\right)\right)\xy
=
\varphi \left([x]_{\e,Y} + \e y, \frac{[x]_{\e,Y} + \e y }{ \e} \right)
=
\varphi ([x]_{\e,Y} + \e y, y) \underset{\e \to 0}{\rightarrow} \varphi\xy
\end{align*}
for every $\xy \in \Omega \times Y$. Since $|\T_\e(\varphi\left(\cdot_x, \frac{\cdot_x}{\e}\right))(x,y)|$ is also pointwise bounded for a.e.~$\xy \in \Omega \times Y$ by $\norm{\varphi}{L^\infty(\Omega \times Y)}$, we can apply Lebesgue's convergence theorem and obtain the strong convergence of $\T_\e(\varphi(\cdot_x, \frac{\cdot_x}{\e}))$ to $\varphi$ in $L^q(\Omega \times Y)$ for $\frac{1}{p}+ \frac{1}{q} = 1$, which implies \ref{thm:T-SC-Equi-Unf}\eqref{item:DistEqui}.

In order to prove Theorem \ref{thm:T-SC-Equi-Unf}\eqref{item:WeakEqui}, we note that both types of weak convergences are equivalent to the boundedness of the sequence plus the corresponding convergence of Theorem \ref{thm:T-SC-Equi-Unf}\eqref{item:DistEqui}.
Using the isometry of $\T_\e$, we can translate the boundedness of the sequences. Then, the equivalence of the weak convergences follows directly from Theorem \ref{thm:T-SC-Equi-Unf}\eqref{item:DistEqui}.

For the equivalence of the strong convergences, we note that the strong convergence of $\T_\e(u_\e)$ is equivalent to the weak convergence of $\T_\e(u_\e)$ plus $\lim\limits_{\e \to 0}\norm{\T_\e(u_\e)}{L^p(\Omega \times Y)} = \norm{u_0}{L^p(\Omega \times Y)}$ since $L^p(\Omega \times Y)$ is a uniformly convex Banach space. Thus, Theorem \ref{thm:T-SC-Equi-Unf}\eqref{item:StrongEqui} follows from Theorem \ref{thm:T-SC-Equi-Unf}\eqref{item:WeakEqui} and the isometry of $\T_\e$.
\end{proof}

Having these results about the unfolding operator, we can prove Theorem \ref{thm:T-SC-Prod} as follows.
\begin{proof}[Proof of Theorem \ref{thm:T-SC-Prod}]
We translate the strong two-scale convergence of $u_\e$ and the weak (resp.~strong) two-scale convergence of $v_\e$ with the unfolding operator $\T_\e$ and Theorem \ref{thm:T-SC-Equi-Unf} into the strong convergence of $\T_\e(u_\e)$  in $L^p(\Omega \times Y)$ and the weak (resp.~strong) convergence of $\T_\e(v_\e)$ in $L^q(\Omega \times Y)$. Employing Hölder estimates, we obtain the weak (resp.~strong) convergence of the product $\T_\e(u_\e v_\e) =\T_\e(u_\e)\T_\e(v_\e)$ to $u_0 v_0$ in $L^r(\Omega \times Y)$. Theorem \ref{thm:T-SC-Equi-Unf} transfers this convergence back into the weak (resp.~strong) two-scale convergences of $u_\e v_\e$ to $u_0v_0 \in L^r(\Omega \times Y)$.
\end{proof}

\section{The two-scale transformation method}\label{sec:TrafoMethod}
In the following, let $Y$ be divided into an open set $\Yp\subset Y$, which constitutes the material part of $Y$, and a hole $Y\setminus \overline{\Yp}$. We assume that the $Y$-periodic extension of $\Yp$, denoted by $Y^*_\#\coloneqq \operatorname{int}\Big(\bigcup\limits_{k \in \Z^N} k + \overline{\Yp} \Big)$, has a Lipschitz boundary.
From now on we assume  that the macroscopic domain $\Omega \subset \R^N$ is not only open and bounded, but also has a Lipschitz boundary. Let $\hOe \coloneqq \Omega \cap \e Y^*_\#$ be the $\e$-scaled periodic reference domains.
Then, we define the locally periodic domains $\Oe$ by transforming the periodic reference domains.
\begin{defi}
We say that a sequence of open domains $\Oe \subset \R^N$ is \textit{locally periodic} if there exists a sequence of locally periodic transformations $\psi_\e$ (see Definition \ref{def:Psi}) such that $\Oe = \psi_\e(\hOe)$.
\end{defi}
In order to give the definition of locally periodic transformations, we have to consider the two-scale convergence for sequences $u_\e$ defined on $\hOe$. For functions defined on $\hOe$, we denote their extension by $0$ on $\Omega$ by $\widetilde{\cdot}$. For functions defined on $\Omega \times \Yp$, we analogously denote their extension by $0$ on $\Omega \times Y$ by $\widetilde{\cdot}$.
\begin{defi}\label{def:Psi}
We say a sequence of $C^1$-diffeomorphisms $\psi_\e :\overline{\hOe} \rightarrow \overline{\Oe}$, for $\Oe \coloneqq \psi_\e(\hOe) \subset\R^N$, is a sequence of \textit{locally periodic transformations} if:
\begin{enumerate}
\item there exists  $c_J >0$ such that $J_\e\geq c_J$ with $J_\e \coloneqq  \det(\Psi_\e)$ and $\Psi_\e \coloneqq D \psi_\e$,
\item there exists $C>0$ such that $\e^{-1}\norm{ \check{\psi}_\e}{C(\overline{\hOe})} + \norm{\nabla \check{\psi}_\e}{C(\overline{\hOe})}\leq C$, where $\check{\psi}_\e(x) \coloneqq \psi_\e(x) -x$ are the corresponding displacement mappings,

\item there exists $\psi_0 \in L^\infty(\Omega;C^1(\overline{\Yp}))^N$, which we call the \textit{limit transformation}, such that 
\begin{enumerate}
\item
$\psi_0(x,\cdot_y) : \overline{\Yp} \rightarrow  \overline{\Ypx}$, for $\Ypx \coloneqq \psi_0(x,\Yp) \subset Y$, are $C^1$-diffeomorphisms for a.e.~$x\in \Omega$,
\item the corresponding displacement mapping, defined by $\check{\psi}_\e(x,y) \coloneqq \psi_\e(x,y) -y$, can be extended $Y$-periodically such that $\check{\psi} \in L^\infty(\Omega;C^1_\#(\overline{\Yp}))$,
\item $\e^{-1}\widetilde{\check{\psi}_\e} \Tscs{\infty} \widetilde{\check{\psi}}_0$ and $\widetilde{\nabla \check{\psi}_\e} \Tscs{\infty} \widetilde{\nabla_y \check{\psi}}_0$.
\end{enumerate}
\end{enumerate}
For a.e.~$x \in \Omega$, we denote the Jacobians of $\psi_0(x,\cdot_y)$ by $\Psi_0(x,\cdot_y) \coloneqq D_y \psi_0(x,\cdot_y)$, $J_0(x,\cdot_y) \coloneqq \det(\Psi_0(x,\cdot_y))$ and its inverse by $\psi_0^{-1}(x,\cdot_y)$. Moreover, we denote the displacement mappings of the back-transformations by $\check{\psi}_\e^{-1}(x) \coloneqq \psi_\e^{-1}(x) -x$ and $\check{\psi}_0^{-1}\xy \coloneqq \psi_0^{-1}\xy -y$.
\end{defi}
We obtain the following uniform estimates and additional strong two-scale convergences as a direct consequence of the definition of the locally periodic transformations $\psi_\e$. 
\begin{lemma}\label{lemma:TwoScalePsi}
Let $\psi_\e$ be locally periodic transformations with limit transformation $\psi_0$. Then, there exist constants $c_J, C >0$ such that
\begin{align*}
&\norm{\Psi_\e}{C(\overline{\hOe})} + \norm{\Psi_\e^{-1}}{C(\overline{\hOe})} + \norm{J_\e}{C(\overline{\hOe})} \leq C,
\\
& \norm{\Psi_0}{L^\infty(\Omega;C(\overline{\Yp}))}+
\norm{\Psi_0^{-1}}{L^\infty(\Omega;C(\overline{\Yp}))}
+
\norm{J_0}{L^\infty(\Omega;C(\overline{\Yp}))}
\leq C, \ \ \ J_0 \geq c_J.
\end{align*}
Furthermore, 
\begin{align*}
\widetilde{\Psi_\e} \Tscs{\infty} \widetilde{\Psi_0}, &&	\widetilde{\Psi_\e^{-1}} \Tscs{\infty} \widetilde{\Psi_0^{-1}}, && \widetilde{J_\e} \Tscs{\infty} \widetilde{J_0}, && \widetilde{J_\e^{-1}} \Tscs{\infty} \widetilde{J_0^{-1}}.
\end{align*}
\end{lemma}
\begin{proof}
The uniform estimate of $\nabla \check{\psi}_\e$ directly gives one for $\Psi_\e = D_x \check{\psi}_\e+ \1$.
Since $J_\e$ and $\Psi_\e^{-1}$ are polynomials in $\Psi_\e$ and $J_\e^{-1}$, the uniform estimates for these follow with the additional uniform bound of $J_\e \geq c_J$ from below.

We rewrite $\widetilde{\Psi_\e} = \widetilde{D \psi_\e} = \widetilde{D \check{\psi}_\e} + \widetilde{Dx|_{\hOe}} = \widetilde{D \check{\psi}_\e} + \1 \chi_{\hOe}$ and $\widetilde{\Psi_0} = \widetilde{D_y \psi_0}= \widetilde{D_y \check{\psi}_0} + \widetilde{D_y y|_{\Yp}} = \widetilde{D_y \check{\psi}_0} + \1 \chi_{\Yp}$. Then, the strong two-scale convergence of $\widetilde{D \check{\psi}_\e}$ to $\widetilde{D_y \psi_0}$ and the strong two-scale convergence of $\chi_\hOe$ to $\chi_\Yp$ directly imply the strong two-scale convergence of $\widetilde{\Psi_\e}$ to $\widetilde{\Psi_0}$ for every $p \in (1,\infty)$.
Note that $\chi_{\hOe}\Tscs{\infty} \chi_\Yp$ because it can be written as $\chi_{\hOe}(x) = \chi_{Y^*_\#}(\xeps)$ for $ \chi_{Y^*_\#} \in L^q_\#(Y;C(\overline{\Omega}))$ for every $q\in (1,\infty)$ (cf. \cite[Theorem 3]{LNW02}).
	
Since $\widetilde{J_\e}$ and $\widetilde{J_0}$ are polynomials with respect to the entries of $\widetilde{\Psi_\e}$ and $\widetilde{\Psi_0}$, respectively, Theorem \ref{thm:T-SC-Prod} implies the strong two-scale convergence of $\widetilde{J_\e}$ to $\widetilde{J_0}$ for every $p \in (1,\infty)$.

The uniform boundedness of $J_\e \geq c_J$ from below gives $\T_\e(\widetilde{J_\e})\xy \geq c_J$  for a.e.~$\xy \in \Omega \times Y$. Then, the strong convergence of $\T_\e(\widetilde{J_\e})$ to $\widetilde{J_0}$ in $L^p(\Omega \times Y)$ transfers the uniform boundedness from below to $J_0\xy \geq c_J$ for a.e.~$\xy \in \Omega \times \Yp$.

We rewrite $\T_\e(\widetilde{J_\e^{-1}}) = (\T_\e(\widetilde{J_\e})\widetilde{|_{\Omega \times \Yp})^{-1}}$ and obtain
\begin{align*}
\norm{\T_\e(\widetilde{J_\e^{-1}}) - \widetilde{J_0^{-1}}}{L^p(\Omega \times Y)} &= \norm{\T_\e(\widetilde{J_\e^{-1}}) - J_0^{-1}}{L^p(\Omega \times \Yp)}= \norm{(J_0 - \T_\e(\widetilde{J_\e}) )/ (J_0 \T_\e(\widetilde{J_\e}))}{L^p(\Omega \times \Yp)}
\\
&\leq \frac{1}{c_J^2}\norm{(J_0 - \T_\e(\widetilde{J_\e})) } {L^p(\Omega \times \Yp)} = \frac{1}{c_J^2}\norm{(\widetilde{J_0} - \T_\e(\widetilde{J_\e})) } {L^p(\Omega \times Y)} \to 0,
\end{align*}
which implies the strong two-scale convergence of $\widetilde{J_\e^{-1}}$ to $\widetilde{J_0^{-1}}$ for any $p \in (1,\infty)$.
	
Since $\widetilde{\Psi_\e^{-1}}$ is a polynomial in $\widetilde{J_\e^{-1}}$ and $\widetilde{\Psi_\e}$, the strong two-scale convergence can be directly transferred to the strong two-scale convergence of $\widetilde{\Psi_\e^{-1}}$ to $\widetilde{\Psi_0^{-1}}$.

Moreover, we obtain $\norm{\Psi_\e^{-1}}{L^\infty(\Omega;C(\overline{\Yp}))} \leq C$ and $ \norm{J_0}{L^\infty(\Omega;C(\overline{\Yp}))} \leq C$ from $\norm{\Psi_0}{L^\infty(\Omega;C(\overline{\Yp}))} \leq C$ and $J_0\xy \geq c_J$ by using the same argumentation as for $J_\e$ and $\Psi_\e^{-1}$.
\end{proof}

For example, a family of diffeomorphisms $\psi_\e$ which are locally periodic transformations in the sense of Definition \ref{def:Psi} can be obtained as follows.
\begin{example}
Let $\Theta : \overline{\Omega} \rightarrow [\Theta_{\textrm{min}}, \Theta_{\textrm{max}}]$ be a continuous function, which describes, for example, the local porosity. Let $\psi : [\Theta_{\textrm{min}}, {\Theta_\textrm{max}}] \times Y \rightarrow Y$ be a smooth mapping such that, for every $\Theta \in[\Theta_{\textrm{min}}, \Theta_{\textrm{max}}]$, $\psi(\Theta, \Yp)$ gives a cell with porosity $\Theta$ and $\psi(\Theta, \cdot) : Y \rightarrow Y$ is a $C^1$-diffeomorphism. Moreover, we assume that there exist $C,c_J >0$ such that $\norm{D_y\psi}{C([\Theta_{\textrm{min}, \textrm{max}}]\times \overline{\Omega})} \leq C$, $\det(D_y\psi) \geq c_J$  and that the corresponding displacement mapping $\check{\psi}(\Theta,y) = \psi(\Theta, y)-y$ has compact support in the interior of $Y$.
Then, $\psi_\e(x) \coloneqq x + \e \check{\psi}\left(\Theta([x]_{\e,Y}), \{x\}_{\e,Y}\right)$ are locally periodic transformations with limit transformation $\psi_0(x,y) \coloneqq \check{\psi}(\Theta(x), y)$ in the sense of Definition \ref{def:Psi}.
\end{example}

Before we continue with the transformation of the two-scale convergence, we recap the main two-scale compactness results for periodic domains.
Since the extension by $0$ does not preserve the $W^{1,q}$-regularity, these compactness results cannot be derived directly from the previous $W^{1,p}$-compactness results (cf. Theorem \ref{thm:T-SC-W1pComp} and Theorem \ref{thm:T-SC-W1pComp2}).
Instead, their derivation extensively utilises the domains' periodic structures.
\begin{prop}\label{prop:PorousCompact}
Let $u_\e$ be a bounded sequence in $L^p(\hOe)$ for $p \in (1,\infty)$. Then, there exists a subsequence $\es$ and $u_0 \in L^p(\Omega \times \Yp)$ such that $\widetilde{u_\es} \TscW{p}\widetilde{u_0}$.
	
Furthermore, let $l \in \{0,2\}$ and let $u_\e$ be a sequence in $W^{1,p}(\hOe)$ such that
\begin{align*}
\norm{u_\e}{L^p(\hOe)} + \e^{l/2}\norm{\nabla u_\e}{L^p(\hOe)} \leq C.
\end{align*}
Then, the following statements hold:
\begin{enumerate}
\item \label{item:PorousCompact}
If $l = 0$ and $Y^*_\#$ is connected, then there exist $u_0 \in L^p(\Omega)$, $u_1 \in L^p(\Omega;W^{1,p}_\#(\Yp)/\R)$ and a subsequence $\es$ such that
$\widetilde{u_\es} \TscW{p}\chi_{\Yp} u_0$ and $\widetilde{\nabla u_\es} \TscW{p}\chi_{\Yp} \nabla_x u_0 + \widetilde{\nabla_y u_1}$.
\item \label{item:PorousCompact2}
If $l = 2$, then there exist $u_0 \in L^p(\Omega; W^{1,p}_\#(\Yp))$  and a subsequence $\es$ such that $\widetilde{u_\es} \TscW{p}\widetilde{u_0}$ and $\es\widetilde{\nabla u_\es} \TscW{p}\widetilde{\nabla_y u_0}$.
\end{enumerate}
\end{prop}
\begin{proof}
Let $u_\e$ be bounded, then $\widetilde{u_\e}$ is bounded as well and Theorem \ref{thm:T-SC-LpComp} gives a subsequence which two-scale converges weakly to a limit function $u_0 \in L^p(\Omega \times Y)$. Employing two-scale test functions $\varphi$ which are $0$ in $\Omega \times \Yp$ yields $u_0 = 0$ in $Y\setminus \Yp$. Thus, we can rewrite the limit as $\widetilde{u_0}$ for $u_0|_{\Omega\times \Yp} \in L^p(\Omega \times \Yp)$.
	
A proof of Proposition \ref{prop:PorousCompact}\eqref{item:PorousCompact} is given in \cite[Theorem 4.6]{All96} for the case $p=2$. It can be generalised to arbitrary $p \in(1,\infty)$ in the same way as the standard $H^1$-two-scale compactness result.

Proposition \ref{prop:PorousCompact}\eqref{item:PorousCompact2} can be proven analogously to Theorem \ref{thm:T-SC-W1pComp2} by using the $L^p(\Oe)$ compactness result of Proposition \ref{prop:PorousCompact} instead of Theorem \ref{thm:T-SC-LpComp}.
\end{proof}

Before we can analyse the two-scale convergence under the two-scale transformation, we have to consider what two-scale convergence of sequences on $\Oe$ means. We note that the definition of the locally periodic transformations $\psi_\e$ does not ensure that $\Oe = \psi_\e(\hOe)$ is contained in $\Omega$. However, it ensures that $\norm{\check{\psi}_\e}{C(\overline{\hOe})}\leq \e C$, which implies that $|\Oe \setminus \Omega| \leq \e C$ as well as $\Oe \subset \{x \in \R^N \mid \operatorname{dist}(x,\Omega) \leq \e C \}$. 
Therefore, we expect a limit defined on the macroscopic domain $\Omega$, which could suggest to formulate the two-scale convergence for functions defined on $\Oe$ by restricting them on $\Omega \cap \Oe$ and then extending them by $0$ on $\Omega$. However, it turns out that this ansatz would not yield a natural translation between the two-scale convergence in the untransformed and the transformed setting.
Instead, we consider $\Od \coloneqq \{x \in \R^N \mid \operatorname{dist}(x,\Omega) < \delta \}$ for fixed $0 < \delta <<1 $ as the macroscopic domain and note that $\Oe \subset \Od$ for $\e$ small enough. We extend functions defined on $\Oe$ by $0$ on $\Od$, which we denote by $\widetilde{\cdot}$. Then, we can use the normal two-scale convergence, but for the macroscopic domain $\Od$ instead of $\Omega$.

However, we will show that the corresponding two-scale limits have support on $Q \coloneqq \{\xy \in \Omega \times Y \mid y \in \Ypx \}$ and the corresponding two-scale limit problems will be defined on $Q$. Therefore, we introduce the following non-cylindrical function spaces on $Q$ for $p\in[1,\infty)$
\begin{align*}
&L^p(\Omega;L^p(\Ypx)) \coloneqq \{f(\cdot_x, \psi_0^{-1}(\cdot_x, \cdot_y)) \mid f \in L^p(\Omega;L^p(\Yp)) \} = L^p(\Omega \times \Yp)
\\
&L^p(\Omega;W^{1,p}_\#(\Ypx)) \coloneqq \{f(\cdot_x, \psi_0^{-1}(\cdot_x, \cdot_y)) \mid f \in L^p(\Omega;W^{1,p}_\#(\Yp)) \}
\end{align*}
with the corresponding norms
\begin{align*}
\norm{u}{L^p(\Omega;L^p(\Ypx))} \coloneqq \norm{u}{L^p(Q)}, \ \ \ 
\norm{u}{L^p(\Omega;W^{1,p}_\#(\Ypx))} \coloneqq 
\norm{u}{L^p(Q)} + \norm{\nabla_y u}{L^p(Q)}.
\end{align*}

Note that the uniform boundedness of the Jacobians $J_0$ and $\Psi_0$ allows to carry all the important functional analytical properties from $L^p(\Omega;W^{1,p}_\#(\Yp))$ over to $L^p(\Omega;W^{1,p}_\#(\Ypx))$ and ensures that the above defined norms are well-defined.
Instead of defining these spaces by a transformation on the cylindrical two-scale reference domain, it would be equivalent to define them directly on the domain $Q$ as $L^p(Q)$ and as the functions in $L^p(Q)$ with $y$-gradient in $L^p(Q)$ and $Y$-periodic trace, respectively. Thus, not only the limit equations but actually also the corresponding weak formulation of the limit problem becomes transformationally independent.

In order to use these function spaces for the two-scale convergence, we extend functions defined on $Q$ by $0$, which we denote by $\widetilde{\cdot}$.
This approach yields an appropriate translation between the two-scale convergence in the untransformed and the transformed setting. 

Note that if $\Oe \subset \Omega$, we do not have to enlarge $\Omega$ and all the following results hold for $\Od = \Omega$.

In order to shorten out notation, we define
\begin{align*}
\varphi_{\e,\psi_\e}(\cdot_x) \coloneqq \varphi\left(\psi_\e(\cdot_x), \frac{\psi_\e(\cdot_x)}{\e}\right),
\ \ \
\varphi_{\psi_0} (\cdot_x,\cdot_y) \coloneqq \varphi(\cdot_x, \psi_0(\cdot_x,\cdot_y)),
\ \ \
\varphi_{\psi_0^{-1}} (\cdot_x,\cdot_y) \coloneqq \varphi(\cdot_x, \psi_0^{-1}(\cdot_x,\cdot_y))
\end{align*}
for functions which depend on $x$ and $y$.
For functions, which have already an index themselves, we write 
$u_{0, \psi_0}(\cdot_x, \cdot_y) = u_0(\cdot_x, \psi_0(\cdot_x, \cdot_y))$ and $\hu_{0, \psi_0^{-1}}(\cdot_x, \cdot_y) = \hu_0(\cdot_x, \psi_0^{-1}(\cdot_x, \cdot_y))$.

First, we consider the two-scale convergence of continuous functions under the locally periodic transformation $\psi_\e$.
\begin{lemma}\label{lem:TwoScaleTrafoTestFunction}
Let $\varphi \in C(\Omega; C_\#(Y))$. Then,
$\widetilde{\pe} \Tscs{\infty} \widetilde{\pO}$.
\end{lemma}
\begin{proof}
Because of Theorem \ref{thm:T-SC-Equi-Unf}, it is enough to show that $\T_\e(\widetilde{\pe})$ converges strongly to $\widetilde{\pO}$ in $L^p(\Omega\times Y)$ for every $p \in (1,\infty)$.
It can be reduced to the strong convergence in $L^p(\Omega\times \Yp)$ since $\operatorname{supp}(\T_\e(\widetilde{\pe}))$, $\operatorname{supp}(\widetilde{\pO)} \subset\Omega \times \Yp$.

Using $\psi_\e(x) = x + \check{\psi}_\e(x)$ and the $Y$-periodicity of $\varphi$, we can rewrite, for a.e.~$\xy \in \Omega \times \Yp$ and $\e>0$ small enough such that $x \in \tilde{\Omega}_\e$
\begin{align*}
\T_\e(\widetilde{\pe}) =
\T_\e&\Big(\widetilde{\varphi\Big(\psi_\e(\cdot_x), \frac{\psi_\e(\cdot_x)}{\e}\Big)}\Big)(x,y) = \varphi\left(\psi_\e( [x ]_{\e,Y} + \e y ), \frac{\psi_\e([x ]_{\e,Y} + \e y )}{\e} \right)
\\
&=
\varphi\left([x ]_{\e,Y} + \e y + \check{\psi}_\e( [x]_{\e,Y} + \e y ), \frac{[x]_{\e,Y} + \e y + \check{\psi}_\e( [x]_{\e,Y} + \e y )}{\e} \right)
\\
&=
\varphi\left([x]_{\e,Y} + \e y + \T_\e(\widetilde{\check{\psi}_\e})(x,y) , y + \frac{\T_\e(\widetilde{\check{\psi}_\e})(x,y)}{\e} \right).
\end{align*}
	
The strong two-scale convergence of $\frac{1}{\e}\widetilde{\check{\psi}_\e}$ to $\widetilde{\check{\psi}_0}$ implies the strong convergence of $\frac{1}{\e} \T_\e(\widetilde{\check{\psi}_\e})$ to $\widetilde{\check{\psi}_0}$ in $L^p(\Omega \times Y)$. Then, there exists a subsequence $\es$ such that $\frac{1}{\es}\T_\es(\widetilde{\check{\psi}_\es})\xy \rightarrow \widetilde{\check{\psi}_0}\xy$ for a.e.~$\xy \in \Omega \times Y$. Moreover, $[x]_{\e,Y}$ converges to $x$ and $\e y$ to $0$. Since $\varphi \in C(\Omega; C_\#(Y))$, we can carry over these pointwise convergences to the pointwise convergence
\begin{align*}
\varphi\left([x]_{\es,Y} + \es y +\T_\es(\widetilde{\check{\psi}_\es})\xy , y + \frac{\T_\es(\widetilde{\check{\psi}_\es})(x,y)}{\es} \right) \to \varphi(x, y + \check{\psi}_0(x,y)) = \varphi(x, \psi_0(x,y))
\end{align*}
for every $\xy \in \Omega \times \Yp$. Furthermore, $\Big|\varphi\Big(\T_\e(\widetilde{\psi_\e})(x,y), y+ \frac{\T_\e(\widetilde{\check{\psi}_\e})(x,y)}{\e} \Big) \Big| \leq \norm{\varphi}{C^\infty(\Omega \times Y)}$ for a.e.~$\xy \in \Omega \times \Yp$. Thus, we can apply Lebesgue's dominated convergence theorem and get the strong convergence of $\T_\es(\widetilde{\varphi_{\es,\psi_\es}})$ to $\widetilde{\pO}$ in $L^p(\Omega\times \Yp)$.
Indeed, this argumentation holds for every arbitrary subsequence, too, which implies the strong convergence for the whole sequence.
\end{proof}
The next lemma shows that the transformations by $\psi_\e$ and $\psi_\e^{-1}$ are uniformly continuous. Together with Lemma \ref{lem:TwoScaleTrafoTestFunction}, this allows to translate between the weak two-scale convergence of sequences defined on $\Oe$ and $\hOe$, respectively.
\begin{lemma}\label{lem:TrafoBounded}
Let $p \in (1,\infty)$. Let $u_\e$ and $\hu_\e = u_\e \circ \psi_\e$ be sequences of measurable functions on $\Oe$ and $\hOe$, respectively. Then, the following statements hold:
\begin{enumerate}
\item \label{enum:LpBoundedTrafo}The sequence $u_\e$ is  bounded in $L^p(\Oe)$ 
if and only if the sequence 
$\hu_\e$ is  bounded in $L^p(\hOe)$.
	
\item \label{enum:GradientLpBoundedTrafo}
Let $l \in \{0,2\}$. Then the sequence $\e^{l/2} \nabla u_\e$ is bounded in $L^p(\Oe)$ 
if and only if the sequence $\e^{l/2}\nabla\hu_\e$ is bounded in $L^p(\hOe)$.
\end{enumerate}
\end{lemma}
\begin{proof}

Transforming the following integrals by $\psi_\e$ and using the uniform estimates on $J_\e$ and $J_\e^{-1}$ implies Lemma \ref{lem:TrafoBounded}\eqref{enum:LpBoundedTrafo}
\begin{align*}
\norm{u_\e}{L^p(\Oe)}^p &= \intOe |u_\e(x)|^p dx
=
\inthOe J_\e(x)|\hu_\e(x)|^p dx \leq C \inthOe |\hu_\e(x)|^p dx = C \norm{\hu_\e}{L^p(\hOe)}^p,
\\
\norm{\hu_\e}{L^p(\hOe)}^p &= \inthOe |\hu_\e(x)|^p dx = \intOe J_\e^{-1}(\psi_\e^{-1}(x)) |u_\e(x)|^p dx \leq c_J \intOe |u_\e(x)|^p dx = c_J \norm{u_\e}{L^p(\Oe)}^p.
\end{align*}

From the chain rule, we get $(\nabla u_\e)(\psi_\e(x)) = \Psi_\e^{-\top}(x) \nabla \hu_\e(x)$. Using the uniform estimates for the Jacobians and their inverses (cf.~Definition \ref{def:Psi} and Lemma~\ref{lemma:TwoScalePsi}), we can estimate as follows
\begin{align*}
\norm{\nabla u_\e}{L^p(\Oe)}^p &= \intOe |\nabla u_\e(x)|^p dx = \inthOe J_\e(x) |\Psi_\e^{-\top}(x) \nabla \hue(x)|^p dx
\leq
C \inthOe \norm{\Psi_\e^{-\top}(x)}{}^p |\nabla \hue(x)|^p dx
\\
&\leq
C \inthOe \norm{\Psi_\e^{-\top}(x)}{L^\infty(\hOe)}^p |\nabla \hue(x)|^p dx
\leq
 C \inthOe |\nabla \hue(x)|^p dx= C \norm{\nabla \hue}{L^p(\hOe)}^p,
\\
\norm{\nabla \hue}{L^p(\hOe)}^p &= \inthOe |\nabla \hue(x)|^p dx = \intOe J_\e^{-1}(\psi_\e^{-1}(x)) |\Psi_\e^{\top}(\psi_\e^{-1}(x)) \nabla u_\e(x)|^p dx 
\\
&\leq
c_J \intOe \norm{\Psi_\e^{\top}(\psi_\e^{-1}(x))}{}^p |\nabla u_\e(x)|^p dx
\leq
c_J \intOe \norm{\Psi_\e^{\top}}{L^\infty(\hOe)}^p |\nabla u_\e(x)|^p dx
\\&\leq
C \intOe |\nabla u_\e(x)|^p dx= C \norm{\nabla u_\e}{L^p(\Oe)}^p,
\end{align*}
which yields Lemma \ref{lem:TrafoBounded}\eqref{enum:GradientLpBoundedTrafo}.
\end{proof}

Now, we give a rigorous translation between the weak two-scale convergence of sequences defined on $\Oe$ and the corresponding sequences defined on $\hOe$. First, we prove the following back-transformation, which was proposed in \cite{Pet07a}.
\begin{thm}\label{thm:TrafoLp}
Let $p \in (1,\infty)$. Let $u_\e$ be a sequence in $L^p(\hOe)$ and $\hue = u_\e \circ \psi_\e$.  Then, $\widetilde{u_\e}\TscW{p}\widetilde{u_0}$ for $u_0 \in L^p(\Omega; L^p(\Ypx))$ if and only if $\widetilde{\hue} \TscW{p}\widetilde{\hu_0}$ for $\hu_0 \in L^p(\Omega \times \Yp)$. Moreover, $\hu_0= u_{0, \psi_0}$ holds and equivalently $u_0 = \hu_{0, \psi_0^{-1}}$.
\end{thm}
\begin{proof}
First, we assume that $\widetilde{\hu_\e}$ two-scale converges to $\widetilde{\hu_0}$ in $L^p(\Omega)$ for $1<p<\infty$. Proposition \ref{prop:T-SC-Bound} implies that $\widetilde{\hu_\e}$ is bounded and by Lemma \ref{lem:TrafoBounded}\eqref{enum:LpBoundedTrafo}, $\widetilde{u_\e}$ is bounded as well.  Moreover, $\widetilde{\hu_{0,\psi_0^{-1}}} \in L^p(\Od \times Y)$.
Therefore, it is enough to show the distributional two-scale convergence, i.e.
\begin{align}\label{eq:proofTrafoLp}
\lim\limits_{\e \to 0}\int\limits_\Od \widetilde{u_\e}(x) \varphi\xxeps dx
=
\int\limits_\Od \intY \widetilde{\hu_{0,\psi_0^{-1}}}\xy \varphi\xy dy dx
\end{align}
for every smooth function $\varphi \in D(\Od; C^\infty_\#(Y))$. We transform the integrand of the left-hand side by $\psi_\e$
\begin{align*}
\int\limits_\Od \widetilde{u_\e}(x) \varphi\xxeps dx
=
\intOe u_\e(x) \varphi\xxeps dx
=
\inthOe J_\e(x) \hu_\e(x) \pe (x) dx
=
\intO \widetilde{J_\e}(x) \widetilde{\hu_\e}(x) \widetilde{\pe}(x) dx.
\end{align*}
We note that $\varphi \in C(\Omega;C_\#(Y))$ and Lemma \ref{lem:TwoScaleTrafoTestFunction} implies $\widetilde{\pe} \TscS{\infty}\widetilde{\pO}$.
Using the strong two-scale convergence $\widetilde{J_\e} \Tscs{\infty} \widetilde{J_0}$, we can pass to the limit $\e \to0$
\begin{align*}
\lim\limits_{\e \to 0} \intO \widetilde{J_\e}(x) \widetilde{\hu_\e}(x) \widetilde{\pe}(x) dx
=
\intOY \widetilde{J_0}\xy \widetilde{\hu_0}\xy \widetilde{\pO}\xy dydx.
\end{align*}
Then, we transform the $Y$-integral back with $\psi_0(x,\cdot_y)$
\begin{align*}
\intOY \widetilde{J_0}\xy \widetilde{\hu_0}\xy \widetilde{\pO}\xy dydx
=
\intOYp J_0\xy \hu_0\xy \varphi(x, \psi_0(x,y)) dydx
\\
=
\intOYpx \hu_0(x, \psi_0^{-1}\xy) \varphi\xy dydx
=
\int\limits_{\Od} \! \intY \widetilde{\hu_{0, \psi_0^{-1}}}\xy \varphi\xy dydx.
\end{align*}
Combining these equations shows \eqref{eq:proofTrafoLp}.

Now, we assume that $\widetilde{u_\e} \TscW{p}\widetilde{u_0}$. By using Proposition \ref{prop:T-SC-Bound} and Lemma \ref{lem:TrafoBounded}\eqref{enum:LpBoundedTrafo}, we obtain the boundedness of $\hu_\e$. Then, Proposition \ref{prop:PorousCompact} gives the existence of a subsequence $\es$ and a function $\hu_0 \in L^p(\Omega \times \Yp)$ such that $\widetilde{\hu_\es}\TscW{p}\widetilde{\hu_0}$. The previous argumentation applied to this subsequence yields $\hu_0 = u_{0,\psi_0}$. Since this argumentation also holds for every subsequence, the whole sequence $\widetilde{\hu_\e}$ two-scale converges weakly to $\widetilde{\hu_0}$ for $\hu_0 = u_{0,\psi_0}$.
\end{proof}

Note that Theorem \ref{thm:TrafoLp} transfers $\chi_\hOe \TscW{\infty} \chi_\Yp$ into $\chi_\Oe \Tscw{\infty} \chi_Q$. Thus, $Q$ and $\Ypx$ can be defined by the two-scale limit of $\chi_\Oe$, which shows that $Q$ and $\Ypx$ are independent of the choice of the periodic reference domain $\hOe$ and the transformations $\psi_\e$ and $\psi_0$.

In the next step, we consider the weak two-scale convergence for weakly differentiable functions. We start with the case of large gradients, i.e.~$\e \nabla u_\e \leq C$, and show that the same transformation rule as for the functions themselves hold.
\begin{thm}\label{thm:TrafoWp2} 
Let $p \in (1,\infty)$. Let $u_\e$ be a sequence in $W^{1,p}(\Oe)$ and $\hue = u_\e \circ \psi_\e$ a sequence in $W^{1,p}(\hOe)$ such that $u_\e$ is bounded in $L^p(\Oe)$. 
Then, $\e \widetilde{\nabla u_\e} \TscW{p}\widetilde{\nabla_y u_0}$ for $u_0 \in L^p(\Omega; W^{1,p}_\#(\Ypx))$ 
if and only if $\e\widetilde{\nabla \hue} \TscW{p}\widetilde{\nabla_y \hu_0}$ for $\hu_0 \in L^p(\Omega; W^{1,p}_\#(\Yp))$. Moreover,  $\hu_0 = u_{0,\psi_0}$ holds and equivalently $u_0 = \hu_{0,\psi_0^{-1}}$.
\end{thm}
\begin{proof}
First, we assume that $\e \widetilde{\nabla \hu_\e}  \TscW{p} \widetilde{\nabla_y \hu_0}$.
Proposition \ref{prop:T-SC-Bound} implies that $\e  \widetilde{\nabla\hu_\e}$ is bounded and by Lemma \ref{lem:TrafoBounded}\eqref{enum:GradientLpBoundedTrafo}, $\e \widetilde{\nabla u_\e}$ is bounded as well. Moreover, $\!\!\!\!\!\! \widetilde{ \ \ \ \nabla_y \hu_{0,\psi_0^{-1}}} \in L^p(\Od \times Y)$.
Therefore, it is enough to show the distributional two-scale convergence, i.e.

\begin{align}\label{eq:proofTrafoWp2}
\lim\limits_{\e \to 0}\int\limits_\Od \e\widetilde{\nabla u_\e}(x) \cdot \varphi\xxeps dx
=
\int\limits_\Od \intY \!\!\!\!\!\! \widetilde{ \ \ \ \nabla_y \hu_{0,\psi_0^{-1}}}\xy \cdot \varphi\xy dy dx
\end{align}
for test functions $\varphi \in D(\Omega;C^\infty_\#(Y))^N$, where $\nabla_y \hu_{0,\psi_0^{-1}}$ denotes the gradient of $y \mapsto \hu_0(x,\psi_0^{-1}(x,y))$.
We transform the integral on the left-hand side by $\psi_\e$ and use the chain rule, which gives $\nabla u_\e(\psi_\e^{-1}(x)) = \Psi_\e^{-\top}(x) \nabla \hu_\e(x)$. Thus, we get
\begin{align*}
\int\limits_\Od \e\widetilde{\nabla u_\e}(x) \cdot\varphi\xxeps dx = \intO \widetilde{J_\e}(x) \widetilde{\Psi_\e^{-\top}}(x)\e\widetilde{\nabla \hue}(x) \cdot \widetilde{\pe}(x)dx.
\end{align*}
In order to pass to the limit $\e \to 0$, we proceed as in the proof of Theorem \ref{thm:TrafoLp} and additionally use the strong two-scale convergence of $\widetilde{\Psi_\e^{-\top}}$. In the limit, we transform the $Y$-integral back  with $\psi_0^{-1}$ and obtain
\begin{align*}
&\lim\limits_{\e \to 0}\int\limits_{\Od} \e \widetilde{\nabla u_\e}(x) \cdot  \varphi\xxeps dx
=
\lim\limits_{\e \to 0}\intO \widetilde{J_\e}(x)\widetilde{\Psi_\e^{-\top}}(x) \e\widetilde{\nabla \hu_\e}(x) \cdot \widetilde{\pe}(x)dx
\\
&=
\intOY \widetilde{J_0}\xy\widetilde{\Psi_0^{-\top}}\xy \widetilde{\nabla_y \hu_0}\xy \cdot \widetilde{\pO}\xy dy dx
=
\int\limits_{\Od} \intY \!\!\!\!\!\! \widetilde{ \ \ \ \nabla_y \hu_{0,\psi_0^{-1}}}\xy \cdot \varphi\xy dy dx.
\end{align*}

Now, we assume that $\e \widetilde{\nabla u_\e} \TscW{p}\widetilde{\nabla_y u_0}$. Using Proposition \ref{prop:T-SC-Bound} and Lemma \ref{lem:TrafoBounded}\eqref{enum:GradientLpBoundedTrafo}, we obtain the boundedness of $\e \nabla \hu_\e$. Moreover, Proposition \ref{lem:TrafoBounded}\eqref{enum:LpBoundedTrafo} transfers the boundedness of $u_\e$ in $L^p(\Oe)$ to the boundedness of $\hu_\e$ in $L^p(\hOe)$.
Then, Proposition \ref{prop:PorousCompact} implies the existence of $\hu_0 \in L^2(\Omega;H^1_\#(\Yp))$ and a subsequence $\es$ such that $\es \widetilde{\nabla u_\es}\TscW{p}\widetilde{\nabla_y \hu_0}$. The previous argumentation applied to this subsequence yields $\hu_0= u_{0,\psi_0}$. Since this argumentation holds for every subsequence, it holds for the whole sequence.
\end{proof}

The last part of the two-scale transformation is the case of small gradients, i.e.~\mbox{$\norm{\nabla u_\e}{L^p(\Oe)} \leq C$}. Following the approach of the case of large gradients $\Psi_0^{-\top}\xy \nabla_x \hu_0(x) + \Psi_0^{-\top}\xy \nabla_y \hu_1\xy$, which has to be transformed back. However, the Jacobian $\Psi_0^{-\top}$ only vanishes by the back-transformation of the $y$-gradient and remains in front of the $x$-gradient. This remaining Jacobian is basically the reason why the back-transformation did not yield a transformationally independent limit problem in the hitherto existing works.
In order to overcome this problem, we separate the purely macroscopic part of $\Psi_0^{-\top}(x) \nabla_x \hu_0(x)$ and put the remaining part into the transformation rule for the $y$-gradient. Thus, we can prove the following new transformation rule.
\begin{thm}\label{thm:TrafoWp0}
Let $p \in (1,\infty)$ and assume that $Y^*_\#$ is connected. Let $u_\e$ be a sequence in $W^{1,p}(\Oe)$ and $\hue = u_\e \circ \psi_\e$ a sequence in $W^{1,p}(\hOe)$ such that $u_\e$ is bounded in $L^p(\Oe)$.
Then, $\widetilde{\nabla u_\e} \TscW{p}\chi_Q \widetilde{\nabla_x u_0} + \widetilde{\nabla_y \hu_1}$ for $u_0 \in W^{1,p}(\Omega)$ and $u_1 \in L^p(\Omega; W^{1,p}_\#(\Ypx)/\R)$ if and only if $\widetilde{\nabla \hue} \TscW{p} \chi_\Yp \nabla_x\hu_0 + \widetilde{\nabla_y \hu_1}$ for $\hu_0 \in W^{1,p}(\Omega)$ and $\hu_1 \in L^p(\Omega; W^{1,p}_\#(\Yp)/\R)$. Moreover, $\hu_0 = u_0$ holds and also $u_1 = \hu_{1,\psi_0^{-1}} + \check{\psi}_0^{-1} \cdot \nabla_x \hat{u}_0$, which is equivalent to $\hu_1 = u_{1,\psi_0} + \check{\psi}_0\cdot \nabla_x u_0$.
\end{thm}
\begin{proof}
First, we assume that $\widetilde{\nabla \hu_\e}\TscW{p} \chi_{\Yp}\hu_0 + \widetilde{\nabla_y \hu_1}$. 
Proposition \ref{prop:T-SC-Bound} implies that $\widetilde{\nabla\hu_\e}$ is bounded and by Lemma \ref{lem:TrafoBounded}\eqref{enum:GradientLpBoundedTrafo}, $\widetilde{\nabla u_\e}$ is bounded as well.  Moreover, $\widetilde{\nabla_y u_1} \in L^p(\Od \times Y)$ for $u_1 = \hu_{1,\psi_0^{-1}} + \check{\psi}_0^{-1} \cdot \nabla_x \hat{u}_0$. Therefore, it is enough to show the distributional two-scale convergence, i.e.
\begin{align}\label{eq:proofTrafoWp0}
\lim\limits_{\e \to 0}&\int\limits_\Od \widetilde{\nabla u_\e}(x) \cdot \varphi\xxeps dx =
\int\limits_\Od \intY \Big(\chi_Q\xy \widetilde{\nabla_x \hu_0}(x) + \widetilde{\nabla_y u_1} \xy \Big) \cdot \varphi\xy dy dx
\end{align}
for $u_1 = \hu_{1,\psi_0^{-1}} + \check{\psi}_0^{-1} \cdot \nabla_x \hat{u}_0$ and 
for test functions $\varphi \in D(\Omega;C^\infty_\#(Y))^N$.
We transform the integral on the left-hand side and pass to the limit $\e \to 0$ like in the proof of Theorem~\ref{thm:TrafoWp2}. After transforming the $Y$-integral back, we use $\Psi^{-\top}_{0,\psi_0^{-1}}\xy = \Psi_0^{-\top}(x,\psi_0^{-1}\xy) = \nabla_y \psi_0^{-1}\xy = \1 + \nabla_y \check{\psi}_0^{-1} \xy$ 
\begin{align*}
\lim\limits_{\e \to 0}&\int\limits_\Od \widetilde{\nabla u_\e}(x) \cdot \varphi\xxeps dx
=
\lim\limits_{\e \to 0}\intO \widetilde{J_\e}(x) \widetilde{\Psi_\e^{-\top}}(x) \widetilde{\nabla \hu_\e}(x) \cdot \widetilde{\pe}(x) dx
\\
&=
\intOY \widetilde{J_0}\xy \widetilde{\Psi_0^{-\top}}\xy (\chi_\Yp(x)\nabla_x \hu_0(x) + \widetilde{\nabla_y \hu_1}\xy ) \cdot \widetilde{\pO}\xy dy dx
\\
&=
\int\limits_\Od \intY \Big(\widetilde{\Psi^{-\top}_{0,\psi_0^{-1}}}\xy \chi_Q\xy \widetilde{\nabla_x \hu_0}(x) + \hspace{-0.45cm} \widetilde{\ \ \ \ \nabla_y \hu_{1,\psi_0^{-1}}}\xy \Big) \cdot \varphi\xy dy dx
\\
&=
\int\limits_\Od \intY \Big(\chi_Q\xy \widetilde{\nabla_x \hu_0}(x) + \widetilde{\nabla_y \check{\psi}_0^{-1}} \xy \widetilde{\nabla_x \hu_0}(x) + \hspace{-0.45cm} \widetilde{\ \ \ \ \nabla_y \hu_{1,\psi_0^{-1}}}\xy \Big) \cdot \varphi\xy dy dx,
\end{align*}
where $\nabla_y \hu_{1,\psi_0^{-1}}$ denotes the gradient of $y \mapsto \hu_1(x, \psi_0^{-1}\xy)$. Thus, $\widetilde{\nabla u_\e} \TscW{p}\chi_Q\widetilde{\nabla_x u_0} + \widetilde{\nabla_y u_1}$ with $u_0=\hu_0$ and $u_1 = \hu_{1,\psi_0^{-1}} + \check{\psi}_0^{-1} \cdot \nabla_x \hat{u}_0$.

Now, we assume that $\widetilde{\nabla u_\e} \TscW{p}\chi_Q \nabla_x u_0 + \widetilde{\nabla_y u_1}$. Then, we obtain $\widetilde{\hu_\e} \TscW{p}\chi_\Yp\nabla_x \hu_0+ \widetilde{\nabla \hu_1}$ with $\hu_0= u_0$ and $\hu_1 = u_{1, \psi_0} - \check{\psi}^{-1}_{0, \psi_0} \cdot \nabla_x u_0$ by the same argumentation as in the proof of Theorem \ref{thm:TrafoWp2}. Rewriting $\check{\psi}_0^{-1} (x,\psi_0\xy) = \psi_0^{-1} (x,\psi_0\xy) - \psi_0\xy = y -\psi_0\xy = -\check{\psi}_0\xy$ gives $\hu_1 = u_{1,\psi_0}+ \check{\psi}_0 \cdot \nabla_x u_0$.
\end{proof}

With the transformation results Theorem \ref{thm:TrafoLp}, Theorem \ref{thm:TrafoWp2} and Theorem \ref{thm:TrafoWp0}, we can translate the two-scale compactness results for periodic domains Theorem \ref{thm:T-SC-LpComp}, Theorem \ref{thm:T-SC-W1pComp} and Theorem \ref{thm:T-SC-W1pComp2} directly into the following compactness results for locally periodic domains.
\begin{thm}
Let $p\in (1,\infty)$. Let $u_\e$ be a bounded sequence in $L^p(\Oe)$. Then, there exist $u_0 \in L^p(\Omega;L^p(\Ypx))$ and a subsequence $\es$ such that $\widetilde{u_\es} \TscW{p} \widetilde{u_0}$.
\end{thm}
\begin{thm}
Let $p \in (1,\infty)$. Let $u_\e$ be a sequence in $W^{1,p}(\Oe)$ such that $\norm{u_\e}{L^p(\Oe)} \leq C$ and $\e\norm{\nabla u_\e}{L^p(\Oe)} \leq C$. Then, there exist $u_0 \in L^p(\Omega;W^{1,p}_\#(\Ypx))$ and a subsequence $\es$ such that $\widetilde{u_\es} \TscW{p} \widetilde{u_0}$ and $\es\widetilde{\nabla u_\es} \TscW{p} \widetilde{\nabla_y u_0}$.
\end{thm}
\begin{thm}
Let $p \in (1,\infty)$ and assume that $Y^*_\#$ is connected. Let $u_\e$ be a sequence in $W^{1,p}(\Oe)$ such that $\norm{u_\e}{L^p(\Oe)}+ \norm{\nabla u_\e}{L^p(\Oe)} \leq C$. Then, there exist $u_0 \in W^{1,p}(\Omega)$, $u_1 \in L^p(\Omega;W^{1,p}_\#(\Ypx)/\R)$ and a subsequence $\es$ such that $\widetilde{u_\es} \TscW{p} \chi_Q\widetilde{u_0}$ and $\widetilde{\nabla u_\es} \TscW{p} \chi_Q \widetilde{\nabla_x u_0} +\widetilde{\nabla_y u_1}$.
\end{thm}

Now, we consider the transformational behaviour of the strong two-scale convergence. This allows us to translate the strong two-scale convergence of the coefficients into the strong two-scale convergence of the transformed coefficients.
Moreover, the following result can also be used to derive further two-scale compactness results for locally periodic domains from the corresponding compactness results in the periodic domain.
Since we have only assumed the strong two-scale convergence for the locally periodic transformations $\psi_\e$ and not a $L^\infty$-convergence, we cannot expect that  the transformation  transfers the strong two-scale convergence in $L^p$ to the strong two-scale convergence in the same $L^p$-spaces. 
However, we obtain the strong two-scale convergence in $L^q(\Omega)$ for every $q \in (1,p)$.
\begin{thm}\label{thm:TrafoStrongTsc}
Let $p \in (1,\infty)$. Let $u_\e$ be a sequence in $L^p(\Oe)$ and $u_0\in L^p(\Omega;L^p(\Ypx))$ such that $\widetilde{u_\e} \TscW{p} \widetilde{u_0}$. Let $\hu_\e = u_\e \circ \psi_\e$ be a sequence in $L^p(\hOe)$ such that $\widetilde{\hue} \TscW{p} \widetilde{\hu_0}$ for $\hu_0\in L^p(\Omega\times \Yp)$ with $\hu_0 = u_{0,\psi_0}$. Then, the following statements hold
\begin{enumerate}
\item If $\widetilde{u_\e} \TscS{p} \widetilde{u_0}$, then $\widetilde{\hu_\e} \Tscs{p} \widetilde{\hu_0}$,
\item If $\widetilde{\hu_\e} \TscS{p} \widetilde{\hu_0}$, then $\widetilde{u_\e} \Tscs{p} \widetilde{u_0}$.
\end{enumerate}
\end{thm}
\begin{proof}
Assume that $\widetilde{u_\e} \TscS{p} \widetilde{u_0}$. Then, it is sufficient to show that $\lim\limits_{\e \to 0}\norm{\widetilde{\hu_\e}}{L^q(\Omega)}=\norm{\widetilde{\hu_0}}{L^q(\Omega \times Y)}$ for every $q \in (1,p)$.
We transform via $\psi_\e$ and get
\begin{align*}
\lim\limits_{\e \to 0}\norm{\widetilde{\hu_\e}}{L^q(\Omega)}^q
=
\lim\limits_{\e \to 0} \intO |\widetilde{\hu_\e}(x)|^q dx
=
\lim\limits_{\e \to 0} \int\limits_\Od \widetilde{J_\e^{-1}\circ\psi_\e^{-1}}(x) |\widetilde{u_\e}(x)|^q dx.
\end{align*}
In order to pass to the limit $\e \to0$, we note that $\widetilde{J_\e^{-1}} \Tscs{\infty}\widetilde{J_0^{-1}}$. Then, Theorem \ref{thm:TrafoLp} implies that $\widetilde{J_\e^{-1}\circ\psi_\e^{-1}} \Tscw{\infty} \widetilde{J^{-1}_{0,\psi_0^{-1}}}$.
Next, we rewrite the strong two-scale convergence of $\widetilde{u_\e}$ into the strong convergence of $\T_\e(\widetilde{u_\e})$ in $L^p(\Od \times Y)$. There, we can deduce the strong convergence of $\T_\e(|u_\e|^{q}) = |\T_\e(\widetilde{u_\e})|^{q} $ to $|\widetilde{u_0}|^{q}$ in $L^{r}(\Od \times Y)$ for $r>1$  small enough.
Then, we can pass to the limit and get
\begin{align}
\lim\limits_{\e \to 0} \intO |\widetilde{\hue}(x)|^q dx
=
\lim\limits_{\e \to 0} \int\limits_\Od \widetilde{J_\e^{-1}\circ\psi_\e^{-1}}(x) |\widetilde{u_\e}(x)|^q dx
=
\int\limits_{\Od} \intY \widetilde{J^{-1}_{0,\psi_0^{-1}}}\xy |\widetilde{u_0}\xy|^q dy dx.
\end{align}
Transforming the right-hand side back yields the desired result.

The proof of the other implication follows in the same way.
\end{proof}

\section{Homogenisation on locally periodic domains}\label{sec:HomogenisationProblem}
In order to pass to the limit $\e \to 0$ in \eqref{eq:weakE}, we have to assume that there exists $A_0 \in L^\infty(Q)^{N\times N}$, which is coercive, such that $\widetilde{A_\e} \Tscs{\infty} \widetilde{A_0}$ and that there exists $f_0 \in L^2(\Omega;L^2(\Ypx))$ such that $\widetilde{f_\e} \TscW{2} \widetilde{f_0}$.
Note that it is not necessary to assume that these two-scale limits are $0$ outside of $Q$. It is sufficient to assume only the existence of the two-scale limits $A_0$ and $f_0$. Then, the two-scale compactness results for locally periodic domains ensures that $A_0 = \chi_Q A_0$ and $f_0 = \chi_Q f_0$.
In the following, we assume that $Y^*_\#$ is connected if $l= 0$.

\subsection{The periodic substitute problem}\label{subsec:TrafoProblem}
We transform the coefficients $A_\e$ and the source functions $f_\e$ into $\hat{A}_\e\coloneqq A_\e \circ \psi_\e$ and $\hat{f}_\e \coloneqq f_\e\circ \psi_\e$, respectively.
Lemma \ref{thm:TrafoLp} implies that $\widetilde{\hat{f}_\e} \TscW{2}\widetilde{\hat{f}_0}$ with $\hat{f}_0 \coloneqq f_{0, \psi_0}$ and Lemma \ref{thm:TrafoStrongTsc} implies that $\widetilde{\hat{A}_\e}  \Tscs{\infty}\widetilde{\hat{A}_0}$ with $\hat{A}_0\coloneqq A_{0, \psi_0}$.
Moreover, note that these transformations carry the uniform boundedness and coercivity from $A_\e$ over to $\hat{A}_\e$ as well as from $A_0$ over to $\hat{A}_0$.
Then, the transformation of \eqref{eq:weakE} with $\psi_\e$ gives the following weak form (cf.~Proposition~\ref{prop:E-WeakForm-Equi}):

Find $\hue \in H^1(\hat{\Omega}_\e)$ such that 
\begin{align}\label{eq:weakEHat}
\inthOe \e^l J_\e(x) \Psi_\e^{-1}(x) \hat{A}_\e(x)\Psi_\e^{-\top}(x) \nabla \hue(x) \cdot \nabla \hp(x) + J_\e(x)\hue(x) \hp(x) dx = \inthOe J_\e(x) \hat{f}_\e(x) \hp(x) dx
\end{align}
for every $\hp \in H^1(\hOe)$.

Using the uniform estimates of the transformations, we show the existence and uniqueness of solutions of \eqref{eq:weakEHat} as well as their uniform boundedness.
\begin{prop}\label{prop:ExiHu}
For every $\e>0$, there exists a unique solution $\hue \in H^1(\hOe)$ of the weak form \eqref{eq:weakEHat} such that
\begin{align}\label{eq:UniformEstHu}
\norm{\hat{u}}{L^2(\hOe)} + \e^{l/2}\norm{\hat{u}}{L^2(\hOe)} \leq C.
\end{align}
\end{prop}
\begin{proof}
Using the uniform bounds of the Jacobians of $\psi_\e$, we can estimate
\begin{align*}
\norm{\nabla u}{L^2(\hOe)}^2
\leq\frac{1}{c_J}
\norm{\sqrt{J_\e}\Psi_\e^{\top}\Psi_\e^{-\top}\nabla u}{L^2(\hOe)}^2
\leq
\frac{1}{c_J}\norm{\Psi_\e^{\top}}{C(\hOe)}^2\norm{\sqrt{J_\e}\Psi_\e^{-\top}\nabla u}{L^2(\hOe)}^2 
\\
\leq C \norm{\sqrt{J_\e}\Psi_\e^{-\top}\nabla u}{L^2(\hOe)}^2
\leq 
\frac{C}{\alpha} \inthOe J_\e(x) \Psi_\e^{-1}(x) \hat{A}_\e(x) \Psi_\e^{-\top}(x)\nabla u(x) \cdot \nabla u(x) dx
\end{align*}
for every $u \in H^1(\hOe)$. This implies the $\e$-independent coercivity of the left-hand side of \eqref{eq:weakEHat} in $H^1(\hOe)$
\begin{align}\label{eq:CoerHu}
\inthOe \e^l J_\e(x) \Psi_\e^{-1}(x) \hat{A}_\e(x)\Psi_\e^{-\top}(x)\nabla u(x) \cdot \nabla u(x) + J_\e(x) u(x) u(x) dx \geq C \left(\e^l\norm{\nabla u}{L^2(\hOe)}^2 + \norm{u}{L^2(\hOe)}^2 \right).
\end{align}
Furthermore, the left-hand side of \eqref{eq:weakEHat} can be estimated for every $u,v \in H^1(\hOe)$ with the Cauchy inequality and the uniform estimates of the transformations
\begin{align*}
\inthOe& \e^l J_\e(x) \Psi_\e^{-1}(x) \hat{A}_\e(x)\Psi_\e^{-\top}(x)\nabla v(x) \cdot \nabla u(x) + J_\e(x)v(x) u(x) dx
\\
&\leq 
\e^l C \norm{\sqrt{J_\e}\Psi_\e^{-\top}\nabla v}{L^2(\hOe)}\norm{\sqrt{J_\e}\Psi_\e^{-\top}\nabla u}{L^2(\hOe)}
+C \norm{v}{L^2(\hOe)} \norm{u}{L^2(\hOe)}
\\
&\leq \e^lC\norm{\nabla v}{L^2(\hOe)}\norm{\nabla u}{L^2(\hOe)} + C\norm{v}{L^2(\hOe)} \norm{u}{L^2(\hOe)}.
\end{align*}
This implies the continuity of the left-hand side.

The right-hand side of \eqref{eq:weakEHat} can be estimated with the uniform estimates from Lemma \ref{lemma:TwoScalePsi} 
\begin{align}\label{eq:ContHf}
\inthOe& J_\e(x) \hat{f}_\e(x) \hp(x) dx
\leq
\norm{J_\e \hat{f}_\e}{L^2(\hOe)} \norm{\hp}{L^2(\hOe)} 
\leq
\norm{J_\e}{C(\hOe)} \norm{\hat{f}_\e}{L^2(\hOe)} \norm{\hp}{L^2(\hOe)}
\leq C \norm{\hp}{L^2(\hOe)}.
\end{align}
Note that $\norm{\hat{f}_\e}{L^2(\hOe)}$ is  bounded since $\widetilde{\hat{f}_\e}$ two-scale converges in $L^2(\Omega)$.

These estimates allow us to apply the Theorem of Lax--Milgram, which gives the existence and uniqueness of a solution $\hue \in H^1(\hOe)$.
Combining \eqref{eq:weakEHat} with \eqref{eq:CoerHu} and \eqref{eq:ContHf} for $\hp = \hue$ and employing the Young inequality yield the uniform estimate \eqref{eq:UniformEstHu}.
\end{proof}

\begin{prop}\label{prop:E-WeakForm-Equi}
Let $u_\e \in H^1(\Oe)$ be the solution of \eqref{eq:weakE}
and let $\hue \in H^1(\hOe)$ be the solution of \eqref{eq:weakEHat}.
Then, $\hue = u_\e \circ \psi_\e$.
\end{prop}
\begin{proof}
The Theorem of Lax--Milgram ensures the existence of a unique solution of \eqref{eq:weakE}.
 
Testing \eqref{eq:weakE} with $\hp \circ \psi^{-1}_\e$, gives
\begin{align}
\intOe A_\e(x) \nabla u_\e(x)\cdot \nabla \hat{\varphi}\circ \psi^{-1}_\e(x)+ u_\e(x) \hat{\varphi}\circ\psi^{-1}_\e(x) dx = \intOe f_\e(x) \hat{\varphi} \circ\psi^{-1}_\e(x) dx.
\end{align}
Transforming the integrals with $\psi_\e$ and using the product rule yield
\begin{align}
\inthOe J_\e(x) \Psi_\e^{-1}(x)\hat{A}_\e(x)\Psi_\e^{-\top}(x) \nabla u_\e\circ\psi_\e(x)\cdot \nabla \hat{\varphi}(x) + u_\e\circ\psi_\e(x) \hp(x) dx = \inthOe J_\e(x) f_\e(\psi_\e(x)) \hat{\varphi}(x) dx.
\end{align}
It follows by the uniqueness of the solution of \eqref{eq:weakEHat} that $\hat{u}_\e = u_\e \circ \psi_\e$.
\end{proof}

\subsection{Homogenisation of the periodic substitute problem} \label{sec:SubstitutedProblem-Homogenisation}
In the following, we pass to the homogenisation limit $\e \to 0$ in \eqref{eq:weakEHat} by using the compactness result for periodic domains (cf.~Proposition~\ref{prop:PorousCompact}).
\begin{prop}\label{prop:Two-scale-Limit-Hu0}
Let $l =0$ and let $Y^*_\#$ be connected. Let $\hue$ be the solutions of \eqref{eq:weakEHat} given by Proposition~\ref{prop:ExiHu}. 
Then, $\widetilde{\hu_\e} \TscW{2}\chi_{\Yp}\hu_0$ and $\widetilde{\nabla u_\e} \TscW{2}\chi_{\Yp} \nabla_x \hu_0 + \widetilde{\nabla_y \hu_1}$, where $(\hu_0 , \hu_1) \in H^1(\Omega) \times L^2(\Omega,H^1_\#(\Yp)/\R)$ is the unique solution of
\begin{align}\notag
\intOYp J_0\xy \Psi_0^{-1}\xy \hat{A}_0\xy\Psi_0^{-\top}\xy (\nabla_x \hu_0(x) + \nabla_y \hu_1\xy) \cdot (\nabla_x \hp_0(x) + \nabla_y \hp_1\xy) dy dx
\\\label{eq:weakForm-two-scale-limit-hat0}
+ \intOYp J_0\xy \hu_0(x) \hp_0(x) dy dx
= 
\intOYp J_0\xy \hat{f}_0\xy \hp_0(x) dy dx
\end{align}
for every $(\hp_0, \hp_1) \in H^1(\Omega) \times L^2(\Omega,H^1_\#(\Yp)/\R)$.
\end{prop}
\begin{proof}
Testing \eqref{eq:weakEHat} with $\hp_0 + \e \hp_1\left(\cdot_x, \frac{\cdot_x}{\e}\right)$ for $\hp_0 \in C^\infty(\Omega)$ and $\hp_1 \in C^\infty(\Omega;C^\infty_\#(Y))$ gives
\begin{align*}
\intO \widetilde{J_\e}(x) \widetilde{\Psi_\e^{-1}}(x) \hat{A}_\e(x) \widetilde{\Psi_\e^{-\top}}(x) \widetilde{\nabla \hu_\e}(x) \cdot \left(\nabla_x \hp_0(x) + \e \nabla_x \hp_1\xxeps + \nabla_y \hp_1\xxeps\right)
\\
+
\widetilde{J_\e}(x)\widetilde{\hue}(x) \left(\hp_0(x) + \e \hp\xxeps\right) dx = \intO \widetilde{J_\e}(x) \widetilde{\hat{f}_\e}(x) \left(\hp_0(x) + \e \hp\xxeps\right) dx.
\end{align*}
The uniform estimate of $u_\e$, given by Proposition \ref{prop:ExiHu}, and  the compactness result for periodic domains (cf.~Proposition \ref{prop:PorousCompact}) yield the existence of $(\hu_0,\hu_1) \in H^1(\Omega) \times L^2(\Omega; H^1_\#(\Yp)/\R)$ and a subsequence $\es$ and such that $\widetilde{\hu_\es} \TscW{2}\chi_{\Yp} \hu_0$ and $\widetilde{\nabla\hat{u}_\es}\TscW{2}\chi_{\Yp}\nabla_x \hu_0 + \widetilde{\nabla_y \hu}_1$.	
Then, we pass to the limit $\es \to 0$ and obtain \eqref{eq:weakForm-two-scale-limit-hat0} for smooth test functions. By a density argument, \eqref{eq:weakForm-two-scale-limit-hat0} follows for test functions in $H^1(\Omega) \times L^2(\Omega; H^1_\#(\Yp)/\R)$.

The existence and uniqueness of the solution $(\hu_0 , \hu_1) \in H^1(\Omega) \times L^2(\Omega,H^1_\#(\Yp)/\R)$ follow from the Theorem of Lax--Milgram. The necessary uniform coercivity and continuity estimates of the left-hand side can be proven in a standard way, while the uniform coercivity of $J_0 \Psi_0^{-1} \hat{A}_0\Psi_0^{-\top}$ can be proven like in Proposition \ref{prop:ExiHu}.

Since this argumentation holds for every subsequence, the uniqueness of the solution of \eqref{eq:weakForm-two-scale-limit-hat0} implies that the convergences hold for the whole sequence.
\end{proof}
We rewrite the two-scale limit problem \eqref{eq:weakForm-two-scale-limit-hat0} into the following homogenised problem, which is defined on the cylindrical two-scale domain $\Omega \times \Yp$ and contains the transformation coefficients $J_0$ and $\Psi_0$.
\begin{prop}\label{prop:LimitHu0}
Let $\hu_0 \in H^1(\Omega)$ be the solution of \eqref{eq:weakForm-two-scale-limit-hat0}. Then, it solves
\begin{align}\label{eq:weak-hu0}
\intO \hat{B}_0(x) \nabla \hu_0(x) \cdot \nabla \hp(x) + \Theta(x)\xy \hu_0(x)\hp(x) dx = \intOYp J_0\xy \hat{f}_0\xy dy \ \hp(x) dx
\end{align}
for every $\hp \in H^1(\Omega)$, where $\Theta(x) = \intYp J_0\xy dy$ and $\hat{B}_0 \in L^\infty(\Omega)^{N \times N}$ is given by
\begin{align}\label{eq:hB0}
(\hat{B}_0)_{ij}(x) \coloneqq \intYp J_0\xy \Psi_0^{-1}\xy \hat{A}_0\xy \Psi_0^{-\top}\xy (e_j +\nabla_y \hat{w}_j\xy) \cdot e_i dy
\end{align}
and $\hat{w}_j$ is defined as the unique solution in $L^2(\Omega; H^1_\#(\Yp)/\R)$ such that 
\begin{align}\label{eq:cellProb-hw}
\intOYp J_0\xy \Psi_0^{-1}\xy \hat{A}_0\xy \Psi_0^{-\top}\xy  (\nabla_y\hat{w}_j\xy + e_j) \cdot \nabla_y \hp_1\xy dydx= 0
\end{align}
for every $\hp_1 \in L^2(\Omega; H^1_\#(\Yp)/\R)$.
\end{prop}

\begin{proof}	
Choosing $\hp_0 = 0$ in \eqref{eq:weakForm-two-scale-limit-hat0} yields
\begin{align}
\intOYp J_0\xy \Psi_0^{-1}\xy\hat{A}_0\xy  \Psi_0^{-\top}\xy (\nabla_x \hu_0(x) + \nabla_y \hu_1\xy) \cdot \nabla_y \hp_1\xy dy dx
=
0,
\end{align}
which implies $\hu_1= \sum\limits_{j = 1}^N \partial_{x_j} \hu_0 \hat{w}_j$, where $\hat{w}_j$ is the unique solution of the cell problem \eqref{eq:cellProb-hw}.
Inserting $\hu_1 = \sum\limits_{j = 1}^N \partial_{x_j} \hu_0 \hat{w}_j$ in \eqref{eq:weakForm-two-scale-limit-hat0} and choosing $\hp_1 = 0$ yield \eqref{eq:weak-hu0} for $\hat{B}_0$ given by \eqref{eq:hB0}.
\end{proof}

\begin{prop}\label{prop:Two-scale-Limit-Hu2}
Let $l =2$ and let $\hue$ be the solution of \eqref{eq:weakEHat} given by Proposition \ref{prop:ExiHu}. 
Then, $\widetilde{\hu_\e} \TscW{2}\widetilde{\hu_0}$ and $\e\widetilde{\nabla \hue} \TscW{2}\widetilde{\nabla_y \hu_0}$, where $\hu_0$ is the unique solution in $L^2(\Omega,H^1_\#(\Yp))$ such that
\begin{align}\label{eq:weakForm-two-scale-limit-hat2}
\intOYp \hat{D}_0\xy \nabla_y \hu_0\xy \cdot \nabla_y\hp\xy + J_0\xy \hu_0\xy\hp\xy dy dx= \intOYp J_0\xy \hat{f}_0\xy \hp\xy dy dx
\end{align}
for every $\hp \in L^2(\Omega;H^1_\#(\Yp))$, where $\hat{D}_0 = J_0 \Psi_0^{-1} \hat{A}_0 \Psi_0^{-\top}$.
\end{prop}
\begin{proof}
The uniform estimate of $\hu_\e$, given by Proposition \ref{prop:ExiHu}, and the compactness result for periodic domains (cf.~\ref{prop:PorousCompact})
imply the existence of a subsequence $\es$ and $\hat{u}_0 \in L^2(\Omega; H^1_\#(\Yp))$ such that $\widetilde{\hu_\es} \TscW{2}\widetilde{\hu_0}$ and $\es\widetilde{\nabla\hu_\es}\TscW{2}\widetilde{\nabla_y \hu_0}$.
We test \eqref{eq:weakEHat} with $\nabla_x \hp\left(\cdot_x, \frac{\cdot_x}{\e}\right)$ for $\hp \in C^\infty(\Omega; C^\infty_\#(Y))$. After passing to the limit $\es \to 0$, we obtain \eqref{eq:weakForm-two-scale-limit-hat2} for smooth test functions. Since $C^\infty(\Omega; C^\infty_\#(Y))$ is dense in $L^2(\Omega; H^1_\#(\Yp))$, \eqref{eq:weakForm-two-scale-limit-hat2} holds for any $\hp \in L^2(\Omega;H^1_\#(\Yp))$.

The existence and uniqueness of the solution $\hu_0 \in L^2(\Omega,H^1_\#(\Yp))$ follow from the Theorem of Lax--Milgram. The necessary uniform coercivity and continuity estimates can be proven in a standard way, while the uniform coercivity of $J_0 \Psi_0^{-1} \hat{A}_0\Psi_0^{-\top}$ can be proven like in Proposition \ref{prop:ExiHu}.

Since this argumentation holds for every subsequence, the uniqueness of the solution of \eqref{eq:weakForm-two-scale-limit-hat2} implies that the convergences hold for the whole sequence.
\end{proof}

\subsection{Back-transformation}\label{subsec:backtrafo}
Using Theorem \ref{thm:TrafoLp}, Theorem \ref{thm:TrafoWp0} and Theorem \ref{thm:TrafoWp2}, we can transform the two-scale limit problems back. Thus, we can derive the two-scale limit problems of \eqref{eq:weakE} for $l=0$ and $l=2$. Moreover, these limit problems do not depend on the chosen transformations $\psi_\e$ and $\psi_0$.
\begin{thm}
Let $l = 0$ and let $Y^*_\#$ be connected. Let $u_\e$ be the solution of \eqref{eq:weakE}. Then, $\widetilde{u_\e} \TscW{2}\chi_{Q}\widetilde{u_0}$ and
$\widetilde{\nabla u_\e} \TscW{2} \chi_Q \widetilde{\nabla_x u_0} + \widetilde{\nabla_y u_1}$,
where $(u_0,u_1)$ is the unique solution in $H^1(\Omega) \times L^2(\Omega;H^1_\#(\Ypx)/\R)$ of
\begin{align}\notag
\int\limits_{\Omega} \intYpx A_0\xy (\nabla_x u_0(x) + \nabla_y u_1\xy) \cdot (\nabla_x \varphi_0(x) + \nabla_y \varphi_1\xy) + u_0(x) \varphi_0(x) dy dx
\\\label{eq:weakForm-two-scale-limit-0}
 = \intO \intYpx f_0\xy dy \ \varphi_0(x) dx
\end{align}
for every $(\varphi_0, \varphi_1) \in H^1(\Omega) \times L^2(\Omega;H^1_\#(\Ypx)/\R)$.
\end{thm}
\begin{proof}
Proposition \ref{prop:E-WeakForm-Equi} shows that $\hu_\e =u_\e \circ \psi_\e$, where $\hu_\e$ is the unique solution of \eqref{eq:weak-hu0}. By Theorem~\ref{thm:TrafoLp}, we obtain $\widetilde{u_\e} \TscW{2}\chi_Q \widetilde{u_0}$ and, by Theorem \ref{thm:TrafoWp0}, $\widetilde{\nabla u_\e} \TscW{2}\chi_Q \widetilde{\nabla_x u_0} + \widetilde{\nabla_y u_1}$ for $u_0= \hat{u}_0$ and $u_1= \hu_{1, \psi_0^{-1}} + \check{\psi}_0^{-1}\cdot \nabla_x u_0$, where $\hu_0$ and $\hu_1$ determine the two-scale limits of $\hue$ and $\nabla \hue$. 

Then, we test \eqref{eq:weakForm-two-scale-limit-hat0} by $(\varphi_0, \varphi_{1,\psi_0} + \check{\psi}_0 \cdot\nabla_x \varphi_0 )$ for $(\varphi_0, \varphi_1 ) \in H^1(\Omega) \times L^2(\Omega; H^1_\#(\Ypx))$ and transform the $\Yp$-integral by $\psi_0^{-1}$ so that
\begin{align}\notag
\intOYpx &A_0\xy \Big(\Psi_0^{-\top}(x,\psi_0^{-1}\xy) \nabla_x \hu_0(x) + \nabla_y \hu_1(x,\psi_0^{-1}\xy)\Big)
\\\notag
&\cdot \Big(\Psi_0^{-\top}(x,\psi_0^{-1}\xy) \nabla_x \varphi_0(x) + \nabla_y\big(\varphi_1\xy + \check{\psi}_0(x,\psi_0^{-1}\xy) \cdot\nabla_x \varphi_0(x) \big)\Big) dy dx
\\\label{eq:weakSemiTransformed}
&+ \intOYpx u_0(x) \varphi_0(x) dy dx
= 
\intOYpx f_0\xy \varphi_0(x) dy dx,
\end{align}
where $\nabla_y \check{\psi}_0(x,\psi_0^{-1}\xy)$ denotes the gradient of $y \mapsto \check{\psi}_0(x,\psi_0^{-1}\xy)$.
Using that $\Psi_0^{-\top}(x,\psi_0^{-1}\xy)= \1 + \nabla_y \check{\psi}_0^{-1} \xy$, we can rewrite
\begin{align*}
\Psi_0^{-\top}(x,\psi_0^{-1}\xy) \nabla_x \hu_0(x) + \nabla_y \hu_1(x,\psi_0^{-1}\xy) &= \nabla_x \hu_0(x) + \nabla_y ( \hu_1(x,\psi_0^{-1}\xy) + \check{\psi}_0^{-1} \cdot \nabla_x u_0(x) )
\\
&=
\nabla_x u_0(x) + \nabla_y u_1\xy.
\end{align*}
Employing that $\check{\psi}_0(x,\psi_0^{-1}\xy) =
y - \psi_0^{-1}\xy 
= -\check{\psi}_0^{-1} \xy$, we get
\begin{align*}
&\Psi_0^{-\top}(x,\psi_0^{-1}\xy) \nabla_x \varphi_0(x) + \nabla_y(\varphi_1\xy + \check{\psi}_0(x,\psi_0^{-1}\xy) \cdot\nabla_x \varphi_0(x) )
\\
&=
\nabla_x \varphi_0(x) + \nabla_y(\varphi_1\xy + \check{\psi}_0(x,\psi_0^{-1}\xy) \cdot\nabla_x \varphi_0(x) + \check{\psi}_0^{-1}\xy\cdot\nabla_x \varphi_0(x))
= \nabla_x \varphi_0(x) + \nabla_y \varphi_1\xy.
\end{align*}
Thus, \eqref{eq:weakSemiTransformed} can be simplified to \eqref{eq:weakForm-two-scale-limit-0}.
\end{proof}

From \eqref{eq:weakForm-two-scale-limit-0}, we can derive the following homogenised limit problem, which is defined on $\Omega$ with cell problems defined on $\Ypx$. However, in contrast to previous works, it does not contain Jacobians of the chosen deformation $\psi_0$.

\begin{thm}\label{thm:Limitu0}
Let $u_0\in H^1(\Omega)$ be the solution of \eqref{eq:weakForm-two-scale-limit-hat0}. Then, it solves	\begin{align}\label{eq:weak-u0}
\intO B_0(x) \nabla u_0(x) \cdot \nabla \varphi(x) + \Theta(x) u_0(x)\varphi(x) dx = \intOYpx f_0\xy dy \ \varphi(x) dx
\end{align}
for every $\varphi \in H^1(\Omega)$,
where $\Theta(x) \coloneqq \intYpx 1 dy = |\Ypx|$ 
and
$B_0 \in L^\infty(\Omega)^{N \times N}$ is given by
\begin{align}\label{eq:A0}
(B_0)_{ij}(x) \coloneqq \intYpx \delta_{ij} +\partial_{y_i} w_j\xy dy
\end{align}
and $w_j$ is defined as the unique solution $w_j \in L^2(\Omega; H^1_\#(\Ypx)/\R)$ such that 
\begin{align}\label{eq:cellProb-w}
\intOYpx  (\nabla_y w_j\xy + e_j) \cdot\nabla_y \varphi\xy dydx= 0
\end{align}
for every $\varphi \in L^2(\Omega; H^1_\#(\Ypx)/\R)$.
\end{thm}
\begin{proof}
The proof of Theorem \ref{thm:Limitu0} runs as the proof of Proposition \ref{prop:LimitHu0}.
\end{proof}
Note that $\Theta$ in \eqref{eq:weak-u0} is the same as in \eqref{eq:weak-hu0} and gives the local porosity of the domain. With Lemma~\ref{lemma:TwoScalePsi}, we see that $\Theta$ is bounded from below by $c_J$ and $\Theta$ is obviously bounded from above by~$1$.

The back-transformation of the two-scale limit problem \eqref{eq:weakForm-two-scale-limit-hat2}  in its actual two-scale domain is straightforward and yields the following limit problem.
\begin{thm}\label{thm:Limitu2}
Let $l =2$ and let $u_\e$ be the solutions of \eqref{eq:weakE}. Then, $\widetilde{u_\e} \TscW{2}\chi_Q \widetilde{u_0}$ and
$\e\widetilde{\nabla u_\e}\TscW{2}\widetilde{\nabla_y u_0}$, where $u_0$ is the unique solution of the following weak form.
Find $u_0 \in L^2(\Omega;H^1_\#(\Ypx))$ such that
\begin{align}\label{eq:weakForm-two-scale-limit-2}
\intO \int\limits_{Y^p_x} A_0\xy \nabla_y u\xy\cdot\nabla_y \varphi\xy + u\xy\varphi\xy dy dx = \intO\int\limits_{Y^p_x} f\xy\varphi\xy dy dx
\end{align}
for every $\varphi \in L^2(\Omega;H^1_\#(\Ypx))$.
\end{thm}
Theorem \ref{thm:Limitu2} follows by testing 
\eqref{eq:weakForm-two-scale-limit-hat2} with $\varphi(\cdot_x,\psi_0(\cdot_x,\cdot_y))$ and back-transformation with $\psi_0^{-1}$. 

\section{Direct homogenisation on the locally periodic domains and further comments}
The compactness results for locally periodic domains which we have developed in this work (cf. Theorem~\ref{thm:TrafoLp}, Theorem  \ref{thm:TrafoWp0} and Theorem \ref{thm:TrafoWp2}) allow us to pass directly to the limit $\e \to 0$ in \eqref{eq:weakE}. The argumentation is the same as in the periodic case.
However, the problem which we have considered is only an easy linear problem and uniform a-priori estimates can easily be derived on the locally periodic domain. If the homogenisation of a more difficult problem is considered, for instance the Stokes problem or non-linear problems, the homogenisation can not solely be done with the compactness results derived in this work. 
Nevertheless, the explicit transformation on the periodic domain becomes useful since the derivation of further two-scale compactness results as well as the derivation of uniform estimates can be easier in the strict periodic setting. Indeed, the results of this work allow to transform uniform estimates and other two-scale compactness results back. Therefore, the homogenisation can be done in the locally periodic setting as well.
Moreover, it can be reasonable to transform the limit problem from the actual non-cylindrical two-scale domain to the cylindrical coordinates $\Omega \times \Yp$ in order to derive uniform estimates on the homogenised tensor and the cell problems.

We want to note that the original motivation for this two-scale transformation method originates from problems on evolving microstructures. There, problems are considered on a time interval $S$ and a time dependent domain $\Oe(t)$ (cf. \cite{Pet07a}, \cite{Pet07b}, \cite{Pet09a}, \cite{Pet09b}, \cite{EM17}, \cite{GNP21}).
The two-scale transformation concept is basically the same for these problems since time is only a parameter in the concept of the two-scale convergence. Thus, our results can be carried over to these problems, where the domain $\Oe(t)\coloneqq \psi_\e(t,\hOe)$ is defined with a family of locally periodic transformations $\psi_\e : S \times \hOe \rightarrow \Omega$ which are dependent on time.

\begin{ack}
I would like to thank Malte~A.~Peter for interesting discussions of this subject.
\end{ack}


\begin{thebibliography}{10}
\bibitem{All92}
G.~Allaire.
\newblock Homogenization and two-scale convergence.
\newblock {\em Siam J. Math. Anal.}, 23:1482--1518, 1992.	

\bibitem{All96}
G.~Allaire and M.~Briane.
\newblock Multiscale convergence and reiterated homogenisation
\newblock {\em Proc. Roy. Soc. Edinburgh}, 126A:297--342, 1996.
	
\bibitem{CDG02}
D.~Cioranescu, A.~Damlamian and G.~Griso. \newblock Periodic unfolding and homogenization.
\newblock {\em C. R. Acad. Sci. Paris Sér. 1}, 335:99--104, 2002.
			
\bibitem{CDG08}
D.~Cioranescu, A.~Damlamian and G.~Griso. \newblock The Periodic Unfolding Method in Homogenization.
\newblock {\em SIAM J. Math. Anal.}, 40:1585--1620, 2008.
	
\bibitem{EM17}
M.~Eden and A.~Muntean.
\newblock Homogenization of a fully coupled thermoelasticity problem for a highly heterogeneous medium with a priori known phase transformations.
\newblock {\em Math. Methods Appl. Sci.}, 40:3955--3972, 2017.

\bibitem{GNP21}
M.~Gahn, M.~Neuss-Radu and I.~S.~Pop.
\newblock Homogenization of a reaction-diffusion-advection problem in an evolving micro-domain and including nonlinear boundary conditions.
\newblock {\em J. Differ. Equations}, 289:95--127, 2021.
	
\bibitem{LNW02}
D.~Lukkassen, G.~Nguetseng and P.~Wall.
\newblock Two-scale convergence.
\newblock {\em Int. J. Pure Appl. Math.}, 2:35--86, 2002.

\bibitem{Pet07a}
M.~A.~Peter.
\newblock Homogenisation in domains with evolving microstructure.
\newblock {\em C.~R.~M\'ecanique}, 335:357--362, 2007.
	
\bibitem{Pet07b}
M.~A.~Peter. 
\newblock Homogenisation of a chemical degradation mechanism inducing an evolving microstructure.
\newblock {\em C.~R.~M\'ecanique}, 335:679--684, 2007.
	
\bibitem{Pet09a}
M.~A.~Peter and M.~B\"ohm.
\newblock Coupled reaction--diffusion processes inducing an evolution of the microstructure: Analysis and homogenization.
\newblock {\em Nonlinear Anal.}, 70:806--821, 2009.
	
\bibitem{Pet09b}
M.~A.~Peter and M.~B\"ohm.
\newblock Multiscale Modelling of Chemical Degradation Mechanisms in Porous Media with Evolving Microstructure.
\newblock {\em Multiscale Model. Simul.}, 7:1643--1668, 2009

\bibitem{SP80} E.~Sanchez-Palencia. \newblock {\em Non-homogeneous Media and Vibration Theory}, volume~127 of {\em Lecture Notes in Physics}. \newblock Springer-Verlag, Berlin Heidelberg, 1980.
	
\bibitem{Tar79}
L.~Tartar.
\newblock Convergence of the homogenization process.
\newblock Appendix of \cite{SP80}.




\end{thebibliography}
\end{document}